%% file: main.tex
\documentclass[11pt]{article}

\input{ex_shared}

\begin{document}
	
\maketitle	

\begin{abstract}
We consider the problem of inferring the interaction kernel of stochastic interacting particle systems from observations of a single particle. We adopt a semi-parametric approach and represent the interaction kernel in terms of a generalized Fourier series. The basis functions in this expansion are tailored to the problem at hand and are chosen to be orthogonal polynomials with respect to the invariant measure of the mean-field dynamics. The generalized Fourier coefficients are obtained as the solution of an appropriate linear system whose coefficients depend on the moments of the invariant measure, and which are approximated from the particle trajectory that we observe. We quantify the approximation error in the Lebesgue space weighted by the invariant measure and study the asymptotic properties of the estimator in the joint limit as the observation interval and the number of particles tend to infinity, i.e. the joint large time-mean field limit. We also explore the regime where an increasing number of generalized Fourier coefficients is needed to represent the interaction kernel. Our theoretical results are supported by extensive numerical simulations.
\end{abstract}



\section{Introduction}

Stochastic interacting particle systems find applications in many areas related to social sciences \cite{GPY17}, collective behavior \cite{NPT10}, pedestrian dynamics~\cite{GSW19}, physics \cite{Gol16}, biology \cite{Suz05}, and machine learning \cite{SiS20}. Recently, tremendous progress has been made in the qualitative and quantitative understanding of such systems, from different perspectives: modeling, analysis, quantitative propagation of chaos~\cite{LaL23a,DGP23}, optimal control~\cite{BKP25}, as well as the development of numerical methods for solving the nonlinear, nonlocal SDEs and PDEs that we obtain in the mean-field limit~\cite{GPV20,CdR24,LTZ25}. We refer to~\cite{CLD22a,CLD22b} for a recent review and references to the literature. In addition, several important contributions have been made to the development of efficient inference methodologies for interacting particle systems and their mean-field limit. Statistical inference, learning, data assimilation, and inverse problems for McKean SDEs and McKean--Vlasov PDEs are topics of great current interest. In particular, the wealth of data in real-world problems pushes towards data-driven models, and therefore learning parameters or functions in mathematical models from observed data is fundamental. A partial list of recent activities on inference for mean-field SDEs and PDEs includes kernel methods~\cite{LMT21,LaL22}, constrast functions based on a pseudo-likelihood \cite{AHP23}, maximum likelihood estimation \cite{LMT22,DeH23}, stochastic gradient descent \cite{PRZ25,SKP23}, approximate likelihood based on an empirical approximation of the invariant measure \cite{GeL24}, method of moments \cite{PaZ24}, eigenfunction martingale estimating functions~\cite{PaZ22,PaZ25}, and regularized variational approaches \cite{CEM25}. An important observation upon which our previous work on inference for McKean SDEs is based is that, under the assumption of uniform propagation of chaos, the nonlinear mean-field SDE can be replaced by a linear (in the sense of Mckean) SDE, obtained by calculating the convolution term in the drift with respect to the (unique) invariant measure of the process. A detailed analysis of this approach can be found in~\cite{PaZ25}. Furthermore, a fully nonparametric estimation of the interaction potential has been studied using different methodologies based on deconvolution \cite{ABP24}, data-driven kernel estimators~\cite{DeH22}, the method of moments~\cite{CGL24}, least squares \cite{LaL22}, and projection techniques \cite{CoG23}. Nonparametric inference for the McKean--Vlasov PDE can also be formulated as an inverse PDE problem and a Bayesian approach can be applied~\cite{NPR25}. We also mention the work \cite{LZT19}, where a nonparametric least squares estimator that does not suffer from the curse of dimensionality is proposed to learn interaction kernels in deterministic dynamical systems.

The goal of this paper is to extend the parametric inference methodologies that were developed recently in~\cite{PaZ24,PaZ22} to a semiparametric setting, combining the method of moments~\cite{PaZ24,CGL24} with a spectral-theoretic approach~\cite{GHR04,PaZ22,Nic24,GiW25} based on generalized Fourier expansions. The main idea behind our approach is to expand the interaction kernel into an appropriate orthonormal basis that is tailored to the problem at hand. 

We will consider a system of $N$ interacting particles moving in one dimension
\begin{equation} \label{eq:interacting_particles} 
\begin{aligned}
\d X_t^{(n)} &= - V'(X_t^{(n)}) \dd t - \frac1N \sum_{i=1}^N W'(X_t^{(n)} - X_t^{(i)}) \dd t + \sqrt{2\sigma} \dd B_t^{(n)}, \\
X_0^{(n)} &\sim \nu, \qquad n = 1, \dots, N,
\end{aligned}
\end{equation}
where $t \in [0,T]$, $V,W \colon \R \to \R$ are the confining and interaction potentials, respectively, and $\sigma > 0$ is the diffusion coefficient. Moreover, $\{B_t^{(n)}\}_{n=1}^N$ are standard independent one-dimensional Brownian motions. We note that the method developed in this paper can be, in principle, extended to the multidimensional case, as we discuss in \cref{sec:conclusion}. We assume chaotic initial conditions with distribution $\nu$, which is, of course, independent of the Brownian motions $\{B_t^{(n)}\}_{n=1}^N$. Then, in the limit as $N \to \infty$, and under appropriate assumptions on the confining and interaction potentials, each particle converges in law to the solution of the mean-field McKean SDE
\begin{equation} \label{eq:mean_field}
\begin{aligned}
\d X_t &= - V'(X_t) \dd t - (W' * u(t, \cdot))(X_t) \dd t + \sqrt{2\sigma} \dd B_t, \\
X_0 &\sim \nu,
\end{aligned}
\end{equation}
where $u(t, \cdot)$ is the density of the law of the process $X_t$ with respect to the Lebesgue measure, and solves the McKean--Vlasov PDE
\begin{equation} \label{eq:FP}
\begin{aligned}
\frac{\partial u}{\partial t}(t,x) &= \frac{\partial }{\partial x} \left( (V'(x) + (W' * u(t,\cdot))(x)) u(t,x) \right) + \sigma \frac{\partial^2 u}{\partial x^2}(t,x), \\
u(x,0) \dd x &= \nu(\d x).
\end{aligned}
\end{equation}
Moreover, the density $\rho$ of the invariant measure satisfies the stationary McKean--Vlasov equation
\begin{equation} \label{eq:McKean_Vlasov}
\frac{\d}{\d x} \Big( V'(x) + (W' * \rho)(x)) \rho \Big) + \sigma \frac{\d^2 \rho}{\d x^2} = 0,
\end{equation}
or, equivalently, the self-consistency equation
\begin{equation} \label{eq:invariant_measure}
\rho(x) = \frac1Z e^{-\frac1\sigma \left( V(x) + (W * \rho)(x) \right)} \quad \text{with} \quad Z = \int_\R e^{-\frac1\sigma \left( V(x) + (W * \rho)(x) \right)} \dd x.
\end{equation}
A diagram illustrating the relationships between the equations introduced so far is shown in \cref{fig:flowchart}. As is well known, without convexity assumptions on interaction and/or confining potentials, the McKean SDE \eqref{eq:mean_field} can have multiple stationary states~\cite{Daw83,Tug14}. However, in this work, we place ourselves in the (strict) convexity/uniform propagation of chaos setting~\cite{Mal01,LaL23a} that ensures the uniqueness of stationary states for the McKean-Vlasov dynamics. We also need the initial distribution to have bounded moments of any order. In particular, we make the following assumption.
\begin{assumption} \label{as:unique}
The process $X_t$ that solves the McKean SDE \eqref{eq:mean_field} is ergodic and admits a unique invariant measure with density $\rho$ that satisfies equations \eqref{eq:McKean_Vlasov} and \eqref{eq:invariant_measure}. Moreover, there exists a constant $\widetilde C > 0$, independent of $T$ and $N$, such that for all $t \in [0,T]$, $n = 1, \dots, N$, and $p \ge 1$
\begin{equation}
\left( \E \left[ (X_t^{(n)} - X_t)^2 \right] \right)^{1/2} \le \frac{\widetilde C}{\sqrt{N}}, \qquad \left( \E \left[ \abs{X_t^{(n)}}^p \right] \right)^{1/p} \le \widetilde C \quad \text{and} \quad \left( \E \left[ \abs{X_t}^p \right] \right)^{1/p} \le C,
\end{equation}
where $X_t$ is given by setting the Brownian motion $B_t = B_t^{(n)}$ in equation \eqref{eq:mean_field}.
\end{assumption}
\begin{remark}
\cref{as:unique} is satisfied, for example, in the setting of \cite{Mal01}, where the confining and interaction potentials are strictly convex and convex, respectively. However, we believe that uniform propagation of chaos is not needed for the inference methodology presented in this work and that the analysis can be extended to the setting considered in~\cite{MoR24}. This is confirmed in the numerical experiments that are presented in \cref{sec:numerics}. The analysis of our inference methodology in the presence of phase transitions will be presented elsewhere.
\end{remark}

\begin{figure}
\centering
\begin{tikzpicture}[scale=0.85]
\begin{small}
\draw[] (0,0) rectangle (3,3) node[pos=.5, align=center] {Interacting \\ particle \\ system \\ (equation \eqref{eq:interacting_particles}) \\ $X_t^{(n)}$};
\draw[->] (3,1.5) -- (4.5,1.5) node[midway, above] {$N \to \infty$};
\draw[] (4.5,0) rectangle (7.5,3) node[pos=.5, align=center] {McKean SDE \\ (equation \eqref{eq:mean_field}) \\ $X_t$};
\draw[->] (7.5,1.5) -- (9,1.5) node[midway, above] {$\mathcal L(X_t)$};
\draw[] (9,0) rectangle (12,3) node[pos=.5, align=center] {McKean--Vlasov \\ PDE \\ (equation \eqref{eq:FP}) \\ $u(t,x)$};
\draw[-] (12,1.5) -- (12.5,1.5); 
\draw[-] (12.5,1.5) -- (12.5,3.5);
\draw[->] (12.5,3.5) -- (14,3.5) node[midway, above] {$\frac{\partial u}{\partial t} = 0$}; 
\draw[] (14,2) rectangle (17,5) node[pos=.5, align=center] {Stationary \\ Fokker--Planck \\ (equation \eqref{eq:McKean_Vlasov}) \\ $\rho(x)$};
\draw[-] (12.5,1.5) -- (12.5,-0.5);
\draw[->] (12.5,-0.5) -- (14,-0.5) node[midway, above] {$t \to \infty$}; 
\draw[] (14,-2) rectangle (17,1) node[pos=.5, align=center] {Invariant \\ measure \\ (equation \eqref{eq:invariant_measure}) \\ $\rho(x)$};
\draw[<->] (15.5,1) -- (15.5,2);
\end{small}
\end{tikzpicture}
\caption{Schematic illustration summarizing the key mathematical objects and their relationships.}
\label{fig:flowchart}
\end{figure}

In this work, we consider the problem of learning the interaction kernel $W'$ (or equivalently the interaction potential $W$) by observing a \emph{single particle} from the interacting particle system \eqref{eq:interacting_particles}. The confining potential is assumed to be known. Intuitively, the approach that we present requires only a single trajectory because it is based on the stationary Fokker--Planck equation \eqref{eq:McKean_Vlasov}. Due to uniform propagation of chaos, all particles become indistinguishable as $N$ grows large. Moreover, the system converges exponentially fast to equilibrium, so for large $T$, the law of the particle is close to the unique invariant measure. Identifiability conditions for the interaction kernel, when there is no confining potential and the observation is of the mean-field PDE solution, have been studied in \cite{LaL23b}. However, it is important to note that it is not possible, in general, to identify both the confining and interaction potentials from a single-particle observation. This was already noted in our previous work~\cite{PaZ22,PaZ24}. In this paper, we make this remark more precise; see \cref{pro:identifiability}. We can consider the case where we either observe a continuous trajectory $(Y_t)_{t \in [0,T]}$ or discrete-time samples from it $\{ \widetilde Y_i \}_{i=0}^I$, where
\begin{equation} \label{eq:observations}
Y_t = X_t^{(\bar n)} \qquad \text{and} \qquad \widetilde Y_i = X_{\Delta i}^{(\bar n)},
\end{equation}
for a particle $\bar n$ in \eqref{eq:interacting_particles} and sampling rate $\Delta = T/I$ fixed (low-frequency regime). Clearly, because of the exchangeability of the interacting particle system, it does not matter which particle we are observing. Assuming now that $W' \in L^2(\rho)$ with $\rho$ the unique invariant measure of the mean-field system, we consider the truncated generalized Fourier series expansion
\begin{equation} \label{eq:W_fourier}
W'(x) \simeq \sum_{k=0}^K \beta_k \psi_k(x),
\end{equation}
where $\{ \psi_k \}_{k=0}^\infty$ are orthogonal polynomials with respect to the invariant measure $\rho$ that, in practice, will be approximated using the available observations. We then propose to estimate the coefficients $\{ \beta_k \}_{k=0}^K$ in the expansion by solving a linear system that is obtained by imposing appropriate constraints on the moments of $\rho$. These conditions, in turn, are derived from the stationary Fokker--Planck equation~\eqref{eq:McKean_Vlasov}. An outline of the main steps of this approach is provided in \cref{alg:procedure} in \cref{sec:method}. We note here the link between this part of the proposed statistical inference methodology and the problem of identifying all generators of reversible diffusions--or, equivalently, all Gibbs measures--whose eigenfunctions are orthogonal polynomials with respect to the invariant measure; see~\cite[Section 2.7]{BGL14} and, in particular,~\cite{BOZ21} for details. Our method builds upon and connects our two previous papers \cite{PaZ22,CGL24} on the application of the method of moments and of eigenfunction expansions to inference problems for interacting particle systems, respectively. In \cite{PaZ22} the main limitation is that the drift, interaction, and diffusion functions are assumed to be polynomials, and, therefore, the inference problem is fully parametric. In an interesting recent paper~\cite{CGL24}, this hypothesis is eliminated, and no assumptions are made about the functional form of potentials. However, for their methodology to work, multiple stationary particle trajectories, for example, 4, of the mean-field SDE must be observed. In addition, the basis functions in the generalized Fourier series expansion in \eqref{eq:W_fourier} have to be chosen appropriately a priori. The two main innovations of the work reported in this paper are: (1) we only need to observe a {\bf single} non-stationary particle trajectory of the interacting particle system and not of the mean-field SDE; and (2) we consider basis functions that are purpose-built and tailored to the stochastic interacting particle system. These basis functions are orthogonal polynomials with respect to the invariant measure of the mean-field dynamics.

In the following, we summarize the main contributions of this paper.

\begin{itemize}[leftmargin=*]
\item We extend our previous work on the application of the method of moments to inference problems for stochastic interacting particle systems. In particular, we introduce a semi-parametric methodology for learning the interaction kernel from the observation of a single-particle trajectory.
\item For the semi-parametric representation of the interaction kernel, we use orthogonal polynomials with respect to the invariant measure of the McKean SDE as basis functions for the Fourier series expansion. In particular, we construct purpose-built orthonormal basis functions that are tailored to the problem at hand. These basis functions are calculated numerically using the available particle trajectory observations.
\item We present a detailed convergence analysis of the proposed methodology, computing estimates of the approximation error as a function of the number of data, particles, and basis functions in the Fourier expansion. We provide thus theoretical guarantees for the convergence and accuracy of our method. 
\item We present several numerical experiments that highlight the effectiveness of our inference methodology.
\end{itemize}

\paragraph{Outline.} The rest of the paper is organized as follows. In \cref{sec:ortho} we consider orthogonal polynomials with respect to the invariant measure of the mean-field dynamics, which we approximate using the available observations from a single interacting particle, and in \cref{sec:method} we present a generalized method of moments to infer the interaction kernel. The convergence analysis in the limit of infinite data and particles for both orthogonal polynomials and the kernel estimator is provided in \cref{sec:analysis_ortho,sec:analysis_method}, respectively. Finally, we present different numerical experiments to test our method in \cref{sec:numerics}, and we suggest possible developments in \cref{sec:conclusion}.

\section{Orthonormal polynomials w.r.t. the invariant measure} \label{sec:ortho}

We consider the orthonormal basis consisting of orthogonal polynomials for the weighted space $L^2(\rho)$, where $\rho$ is the invariant measure of the mean-field dynamics. Starting from the set of monomials $\{ x^k \}_{k=0}^\infty$, the orthogonal polynomials $\{ \psi_k \}_{k=0}^\infty$ can be derived using the Gram--Schmidt orthonormalization procedure.
\begin{equation}
\psi_k(x) = \begin{cases}
1 & \text{if } k = 0, \\
\frac{x^k - \sum_{j=0}^{k-1} \psi_j(x) \int_\R y^k \psi_j(y) \rho(y) \dd y}{\norm{x^k - \sum_{j=0}^{k-1} \psi_j(x) \int_\R y^k \psi_j(y) \rho(y) \dd y}_{L^2(\rho)}} & \text{if } k \ge 1.
\end{cases}
\end{equation}
We mention that it is always possible to construct orthogonal polynomials with respect to a given probability measure on the real line, provided that the measure has finite moments of all orders. We refer to~\cite{Sze75} for more details. Notice that, since $\{ \psi_k \}_{k=0}^\infty$ are polynomials, the Gram--Schmidt procedure is only dependent on the moments $\{ \mathbb M^{(r)} \}_{r=0}^\infty$ of the measure $\rho$
\begin{equation}
\mathbb M^{(r)} = \E^\rho[X^r] = \int_\R x^r \rho(x) \dd x,
\end{equation} 
where the superscript $\rho$ denotes the fact that the expectation is computed with respect to the measure $\rho$. In particular, the $k$-th basis function is given by the determinant of a Hankel matrix with an additional row made of monomials 
\begin{equation} \label{eq:det_moments}
\psi_k(x) = \frac1{c_k} \det \left( \begin{bmatrix}
\mathbb M^{(0)} & \mathbb M^{(1)} & \mathbb M^{(2)} & \cdots & \mathbb M^{(k)} \\
\mathbb M^{(1)} & \mathbb M^{(2)} & \mathbb M^{(3)} & \cdots & \mathbb M^{(k+1)} \\
\vdots & \vdots & \vdots & \ddots & \vdots \\
\mathbb M^{(k-1)} & \mathbb M^{(k)} & \mathbb M^{(k+1)} & \cdots & \mathbb M^{(2k-1)} \\
1 & x & x^2 & \cdots & x^k
\end{bmatrix} \right) \eqdef \frac1{c_k} \det(\mathcal M_k),
\end{equation}
where $c_k$ is the normalization constant that ensures $\norm{\psi_k}_{L^2(\rho)} = 1$. Using the Laplace expansion for the determinants, we then write
\begin{equation} \label{eq:ortho_poly_def}
\psi_k(x) = \frac1{c_k} \sum_{j=0}^{k} \lambda_{kj} x^j,
\end{equation}
where
\begin{equation} \label{eq:lambda_def}
\begin{aligned}
\lambda_{kj} &= (-1)^{k+j} \det \left( \begin{bmatrix}
\mathbb M^{(0)} & \cdots & \mathbb M^{(j-1)} & \mathbb M^{(j+1)} & \cdots & \mathbb M^{(k)} \\
\mathbb M^{(1)} & \cdots & \mathbb M^{(j)} & \mathbb M^{(j+2)} & \cdots & \mathbb M^{(k+1)} \\
\vdots & \ddots & \vdots & \vdots & \ddots & \vdots \\
\mathbb M^{(k-1)} & \cdots & \mathbb M^{(k+j-2)} & \mathbb M^{(k+j)} & \cdots & \mathbb M^{(2k-1)}
\end{bmatrix} \right) \\
&\eqdef (-1)^{k+j} \det(\Lambda_{kj}).
\end{aligned}
\end{equation}
Therefore, the normalization constant satisfies
\begin{equation} \label{eq:c_def}
c_k^2 = \int_{\R^d} \left( \sum_{j=0}^{k} \lambda_{kj} x^j \right)^2 \rho(x) \dd x = \sum_{i=0}^k \sum_{j=0}^k \lambda_{ki} \lambda_{kj} \int_{\R^d} x^{i+j} \rho(x) \dd x = \sum_{i=0}^k \sum_{j=0}^k \lambda_{ki} \lambda_{kj} \mathbb M^{(i+j)}.
\end{equation}
We now aim to approximate the basis functions using the available observations $(Y_t)_{t \in [0,T]}$ or $\{ \widetilde Y_i \}_{i=0}^I$. Since for the construction of $\psi_k$ we only need the moments $\{ \mathbb M^{(r)} \}_{r=0}^{2k-1}$, it is sufficient to estimate them using the empirical moments 
\begin{equation} \label{eq:approximated_moments}
\widetilde{\mathbb M}^{(r)}_{T,N} = \frac1T \int_0^T Y_t^r \dd t, \qquad \widetilde{\mathbb M}^{(r)}_{I,N} = \frac1I \sum_{i=1}^I \widetilde Y_i^r,
\end{equation}
depending on whether we have continuous-time or discrete-time observations. We remark that the subscript $N$ indicates that the empirical moments depend (implicitly) on the number of particles in the system, as the observations are obtained from an interacting particle system of size $N$. We can then build an approximation $\{ \widetilde \psi_k \}_{k=0}^\infty$ of the orthonormal basis using formula \eqref{eq:det_moments}, where the exact moments are replaced by the empirical moments. Notice that we denote by $\widetilde \lambda_{kj}$ and $\widetilde c_k$ the coefficients and the normalization constant of the polynomial $\widetilde \psi_k$, respectively.
\begin{remark} \label{rem:discrete}
From now on, for the sake of simplicity, we will only consider the case of continuous-time observations. Nevertheless, all the analysis presented here still holds in the case of discrete-time observations, and therefore the methodology introduced in the next section can still be applied. We emphasize the fact that our approach only relies on the approximation of the moments of the invariant measure of the mean-field dynamics, and, due to the ergodic theorem, it does not matter whether we use discrete-time or continuous-time observations to estimate the moments. To simplify the notation, we will also remove the subscripts $T$ and $N$ when it is clear from the context that the referred quantities depend on the observation time and the number of particles.
\end{remark}
We finally recall the result from \cite[Lemma 4.3]{PaZ24}, which, due to \cref{as:unique}, using ergodicity and propagation of chaos, and assuming the particles to be initially distributed accordingly to the invariant measure of the mean-field dynamics, i.e. $\nu(\d x) = \rho(x) \dd x$, quantifies the error given by the approximation of the moments. In particular, there exists a constant $\widetilde C > 0$ independent of $T$ and $N$, such that for all $q \in [1,2)$
\begin{equation} \label{eq:estimate_moments}
\left( \E \left[ \abs{\widetilde{\mathbb M}^{(r)} - \mathbb M^{(r)}}^q \right] \right)^{\frac1q} \le \widetilde C \left( \frac1{\sqrt T} + \frac1{\sqrt N} \right),
\end{equation} 
and this estimate will be useful throughout the paper. We note that, as explained in~\cite[Remark 4.4]{PaZ24}, the stationarity assumption is made only to simplify the analysis. Due to geometric ergodicity and uniform propagation of chaos, it should be sufficient to start either at a deterministic initial condition or at a distribution that has, for example, finite relative entropy with respect to the invariant measure and finite moments of all orders. In fact, in all the numerical experiments in \cref{sec:numerics}, we do not start at stationarity, and the particle trajectory that we observe and that we use to estimate the moments is not stationary. 

In the next section, we show the convergence of the approximated orthogonal polynomials $\widetilde \psi_k$ to the corresponding $\psi_k$ in $L^2(\rho)$ as $T,N \to \infty$. 

\subsection{Convergence analysis for the approximated orthogonal polynomials} \label{sec:analysis_ortho}

Before presenting the main result of this section, that is, the convergence of $\widetilde \psi_k$ to $\psi_k$, we need the following estimates regarding the coefficients $\widetilde \lambda_{kj}$ of the polynomials and the normalization constants $\widetilde c_k$.

\begin{lemma} \label{lem:estimate_lambda}
Let $\lambda_{kj}$ be defined in equation \eqref{eq:lambda_def}, and let $\widetilde\lambda_{kj}$ be its approximation. Under \cref{as:unique}, for all $k,j \ge 0$ and for all $q \in [1,2)$ there exists a constant $C = C(k) > 0$, independent of $T$ and $N$, such that
\begin{equation}
\left( \E \left[ \abs{\lambda_{kj} - \widetilde \lambda_{kj}}^q \right] \right)^{\frac1q} \le C(k) \left( \frac1{\sqrt T} + \frac1{\sqrt N} \right).
\end{equation}
\end{lemma}
\begin{proof}
First, by \cite[Theorem 2.12]{IpR08}, we have
\begin{equation}
\abs{\lambda_{kj} - \widetilde \lambda_{kj}}^q = \abs{\det(\Lambda_{kj}) - \det(\widetilde \Lambda_{kj})}^q \le k^q \norm{\Lambda_{kj} - \widetilde \Lambda_{kj}}_2^q \max \left\{ \norm{\Lambda_{kj}}_2, \norm{\widetilde \Lambda_{kj}}_2 \right\}^{q(k-1)},
\end{equation}
which, due to the Hölder's inequality with exponent $1/q+1/2 \in (1, 3/2]$, implies
\begin{equation}
\E \left[ \abs{\lambda_{kj} - \widetilde \lambda_{kj}}^q \right] \le k^q \E \left[ \norm{\Lambda_{kj} - \widetilde \Lambda_{kj}}_2^{\frac{q}2 + 1} \right]^{\frac{2q}{q+2}} \E \left[ \max \left\{ \norm{\Lambda_{kj}}_2, \norm{\widetilde \Lambda_{kj}}_2 \right\}^{\frac{q(k-1)(2+q)}{2-q}} \right]^{\frac{2-q}{2+q}},
\end{equation}
where we remark that $q/2+1 \in [3/2,2)$. Then, since the spectral norm is bounded by the Frobenius norm and due to the uniform boundedness of the moments and the estimate \eqref{eq:estimate_moments}, following the proof of \cite[Lemma 4.5(i)]{PaZ24} we obtain the desired result. We finally remark that the dependence of $C$ on $j$ is not explicitly emphasized since $j \le k$.
\end{proof}

\begin{lemma} \label{lem:estimate_c_k}
Let $c_k$ be defined in equation \eqref{eq:c_def}, and let $\widetilde c_k$ be its approximation. Under \cref{as:unique}, for all $k \ge 0$ and for all $q \in [1,2)$ there exists a constant $C = C(k) > 0$, independent of $T$ and $N$, such that
\begin{equation}
\begin{aligned}
\mathrm{(i)}& \quad \left( \E \left[ \abs{c_k^2 - \widetilde c_k^2}^q \right] \right)^{\frac1q} \le C(k) \left( \frac1{\sqrt T} + \frac1{\sqrt N} \right), \\
\mathrm{(ii)}& \quad \left( \E \left[ \abs{c_k - \widetilde c_k}^q \right] \right)^{\frac1q} \le C(k) \left( \frac1{\sqrt T} + \frac1{\sqrt N} \right).
\end{aligned}
\end{equation}
\end{lemma}
\begin{proof}
From the definitions of $c_k$ in equation \eqref{eq:c_def} and of its empirical counterpart $\widetilde c_k$, using Jensen's inequality we deduce that
\begin{equation}
\E \left[ \abs{c_k^2 - \widetilde c_k^2}^q \right] \le C \sum_{i=0}^k \sum_{j=0}^k \E \left[ \abs{\lambda_{ki} \lambda_{kj} \mathbb M^{(i+j)} - \widetilde\lambda_{ki} \widetilde\lambda_{kj} \widetilde{\mathbb M}^{(i+j)}}^q \right],
\end{equation}
which, due to the triangle inequality, implies
\begin{equation}
\begin{aligned}
\E \left[ \abs{c_k^2 - \widetilde c_k^2}^q \right] &\le \sum_{i=0}^k \sum_{j=0}^k \E \left[ \abs{\lambda_{ki} \lambda_{kj} \left( \mathbb M^{(i+j)} - \widetilde{\mathbb M}^{(i+j)} \right)}^q \right] \\
&\quad + \sum_{i=0}^k \sum_{j=0}^k \E \left[ \abs{\lambda_{ki} \left( \lambda_{kj} - \widetilde\lambda_{kj} \right) \widetilde{\mathbb M}^{(i+j)}}^q \right] \\
&\quad + \sum_{i=0}^k \sum_{j=0}^k \E \left[ \abs{ \left( \lambda_{ki} - \widetilde\lambda_{ki} \right) \widetilde\lambda_{kj} \widetilde{\mathbb M}^{(i+j)}}^q \right].
\end{aligned}
\end{equation}
Then, applying Hölder's inequality, the boundedness of the moments, estimate \eqref{eq:estimate_moments}, and \cref{lem:estimate_lambda}, we obtain (i). Notice now that we have
\begin{equation}
\abs{c_k - \widetilde c_k} = \frac{\abs{c_k^2 - \widetilde c_k^2}}{c_k + \widetilde c_k} \le \frac{\abs{c_k^2 - \widetilde c_k^2}}{c_k},
\end{equation}
which due to (i) gives (ii) and completes the proof.
\end{proof}

Notice that \cref{lem:estimate_c_k} gives bounds on the difference between the normalization constants. However, $c_k$ and $\widetilde c_k$ appear in the denominator of orthogonal polynomials. Therefore, we need to obtain an estimate in $L^2$ of the difference of their inverse, for which the following condition on the boundedness of negative moments for $\widetilde c_k$ is required.
\begin{assumption} \label{as:inverse_ck}
For all $k \ge 0$ there exists a constant $C = C(k) > 0$, independent of $T$ and $N$, such that
\begin{equation}
\E \left[ \frac1{\widetilde c_k^r} \right] \le C(k),
\end{equation}
for some $r > 4$.
\end{assumption}
\begin{remark}
\cref{as:inverse_ck} is always satisfied as long as $\widetilde c_k$ is bounded away from zero independently of $T$ and $N$, e.g. $\widetilde c_k \ge \zeta c_k$ for some $\zeta > 0$. Moreover, in all the numerical experiments that we performed, we observed that \cref{as:inverse_ck} holds in practice. We remark that, without this condition, the convergence analysis presented in the following still goes through. However, we could only achieve convergence in probability and without quantitative convergence rates.
\end{remark}
In the next result, we derive the estimate for the inverse of $\widetilde c_k$.

\begin{lemma} \label{lem:estimate_inverse_ck}
Let $c_k$ be defined in equation \eqref{eq:c_def}, and let $\widetilde c_k$ be its approximation. Under \cref{as:unique,as:inverse_ck}, for all $k \ge 0$ and for all $q \in [1,2r/(r+2))$ there exists a constant $C = C(k) > 0$, independent of $T$ and $N$, such that
\begin{equation}
\left( \E \left[ \abs{\frac1{c_k} - \frac1{\widetilde c_k}}^q \right] \right)^{\frac1{q}} \le C(k) \left( \frac1{\sqrt T} + \frac1{\sqrt N} \right).
\end{equation}
\end{lemma}
\begin{proof}
Using Hölder's inequality with exponents $r/q$ and $r/(r-q)$, we obtain
\begin{equation}
\E \left[ \abs{\frac1{c_k} - \frac1{\widetilde c_k}}^q \right] = \E \left[ \frac{\abs{c_k - \widetilde c_k}^q}{c_k^q \widetilde c_k^q} \right] \le \frac1{c_k^q} \left( \E \left[ \abs{c_k - \widetilde c_k}^{\frac{rq}{r-q}} \right] \right)^{\frac{r-q}{r}} \left( \E \left[ \frac1{\widetilde c_k^r} \right] \right)^{\frac{q}{r}},
\end{equation}
where we note that $rq/(r-q) \in [1,2)$, and which implies
\begin{equation}
\left( \E \left[ \abs{\frac1{c_k} - \frac1{\widetilde c_k}}^q \right] \right)^{\frac1{q}} \le \frac1{c_k} \left( \E \left[ \abs{c_k - \widetilde c_k}^{\frac{rq}{r-q}} \right] \right)^{\frac{r-q}{rq}} \left( \E \left[ \frac1{\widetilde c_k^r} \right] \right)^{\frac{1}{r}}.
\end{equation}
The desired result then follows from \cref{as:inverse_ck} and \cref{lem:estimate_c_k}(ii).
\end{proof}

Using the previous estimates, we can finally prove the convergence of the approximated orthogonal polynomials in the weighted space $L^2(\rho)$.

\begin{proposition} \label{pro:estimate_psi}
Let $\psi_k$ be defined in equation \eqref{eq:ortho_poly_def}, and let $\widetilde\psi_k$ be its approximation. Under \cref{as:unique,as:inverse_ck}, for all $k \ge 0$ there exists a constant $C = C(k) > 0$, independent of $T$ and $N$, such that
\begin{equation}
\E \left[ \norm{\psi_k - \widetilde \psi_k}_{L^2(\rho)} \right] \le C(k) \left( \frac1{\sqrt T} + \frac1{\sqrt N} \right).
\end{equation}
\end{proposition}
\begin{proof}
By definition of $\psi_k$ and $\widetilde \psi_k$ we have
\begin{equation}
\norm{\psi_k - \widetilde \psi_k}_{L^2(\rho)}^2 = \int_{\R^d} \left( \frac1{c_k} \sum_{j=0}^k \lambda_{kj} x^j - \frac1{\widetilde c_k} \sum_{j=0}^k \widetilde\lambda_{kj} x^j \right)^2 \rho(x) \dd x,
\end{equation}
which due to the triangle inequality gives
\begin{equation}
\begin{aligned}
\norm{\psi_k - \widetilde \psi_k}_{L^2(\rho)}^2 &\le \frac2{\widetilde c_k^2} \int_{\R^d} \left( \sum_{j=0}^k ( \lambda_{kj} - \widetilde\lambda_{kj} ) x^j \right)^2 \rho(x) \dd x \\
&\quad + 2 \left( \frac1{c_k} - \frac1{\widetilde c_k} \right)^2 \int_{\R^d} \left( \sum_{j=0}^k \lambda_{kj} x^j \right)^2 \rho(x) \dd x.
\end{aligned}
\end{equation}
By expanding the square we get
\begin{equation}
\begin{aligned}
\norm{\psi_k - \widetilde \psi_k}_{L^2(\rho)}^2 &\le \frac2{\widetilde c_k^2} \sum_{i=0}^k \sum_{j=0}^k ( \lambda_{ki} - \widetilde\lambda_{ki} ) ( \lambda_{kj} - \widetilde\lambda_{kj} ) \mathbb M^{(i+j)} \\
&\quad + 2 \left( \frac1{c_k} - \frac1{\widetilde c_k} \right)^2 \sum_{i=0}^k \sum_{j=0}^k \lambda_{ki} \lambda_{kj} \mathbb M^{(i+j)},
\end{aligned}
\end{equation}
which implies
\begin{equation}
\norm{\psi_k - \widetilde \psi_k}_{L^2(\rho)}^2 \le C \frac1{\widetilde c_k^2} \left( \sum_{j=0}^k \abs{\lambda_{kj} - \widetilde\lambda_{kj}} \right)^2 + C \left( \frac1{c_k} - \frac1{\widetilde c_k} \right)^2 \left( \sum_{j=0}^k \abs{\lambda_{kj}} \right)^2.
\end{equation}
Therefore, using Hölder's inequality with exponents $r$ and $q = r/(r-1) \in (1,2)$, we have
\begin{equation}
\begin{aligned}
\E \left[ \norm{\psi_k - \widetilde \psi_k}_{L^2(\rho)} \right] &\le C \E \left[ \frac1{\widetilde c_k} \sum_{j=0}^k \abs{\lambda_{kj} - \widetilde\lambda_{kj}} \right] + C \E \left[ \abs{\frac1{c_k} - \frac1{\widetilde c_k}} \right] \\
&\le C \E \left( \left[ \frac1{\widetilde c_k^r} \right] \right)^{1/r} \left( \sum_{j=0}^k \E \left[ \abs{\lambda_{kj} - \widetilde\lambda_{kj}} \right]^q \right)^{1/q} + C \E \left[ \abs{\frac1{c_k} - \frac1{\widetilde c_k}} \right],
\end{aligned}
\end{equation}
which, due to \cref{as:inverse_ck,lem:estimate_lambda,lem:estimate_inverse_ck}, yields the desired result.
\end{proof}

\section{Generalized method of moments} \label{sec:method}

In this section, we present our inference methodology for learning the interaction kernel $W'$ in \eqref{eq:interacting_particles} from a single particle observation. Let $\{ \psi_k \}_{k=0}^\infty$ be the orthonormal basis made of orthogonal polynomials introduced in the previous section, and let $\{ \widetilde\psi_k \}_{k=0}^\infty$ be its approximation obtained using empirical moments. We now introduce the integrability (with respect to the invariant measure of the mean-field dynamics) and regularity assumptions on the confining and interaction potentials, as well as the approximability property of the weighted $L^2$ space.
\begin{assumption} \label{as:L2_rho}
The confining and interaction potentials satisfy $V, \,W \in \mathcal C^2(\R) \cap H^1(\rho)$, where $H^1(\rho)$ denotes the standard weighted Sobolev space, and the set of polynomials is dense in $L^2(\rho)$. Moreover, $W' * \rho \in L^2(\rho)$ and $V'$ and $V''$ are polynomially bounded. 
\end{assumption}

\begin{remark}
The fact that the set of polynomials is dense together with the orthogonality of the polynomials by construction implies that the set $\{ \psi_k \}_{k=0}^\infty$ forms an orthonormal basis of $L^2(\rho)$. The problem of finding conditions that guarantee that the set of polynomials is dense in weighted $L^2$ spaces has been thoroughly investigated and is directly related to the Hamburger moment problem. Sufficient conditions based on the definition of Nevanlinna extremal (N-extremal) solutions are given in \cite[Section 2.3]{Akh21} and \cite{Ber96}. Moreover, in \cite{Rod03} the analysis is extended to the general case of weighted Sobolev spaces $W^{k,p}(\rho)$, and conditions on the weight $\rho$ are provided.
\end{remark}

Consider the Fourier expansions of both $V'$ and $W'$ with respect to the basis $\{ \psi_k \}_{k=0}^\infty$ made of orthogonal polynomials 
\begin{equation}
V'(x) = \sum_{k=0}^\infty \alpha_k \psi_k(x), \qquad W'(x) = \sum_{k=0}^{\infty} \beta_k \psi_k(x).
\end{equation}
We emphasize that the confining potential is assumed to be known. Our goal is to infer the coefficients $\beta_k$ from approximations of the coefficients $\alpha_k$. In particular, consider the McKean--Vlasov PDE \eqref{eq:McKean_Vlasov}, and replace $V'$ and $W'$ with their Fourier expansions
\begin{equation}
\frac{\d}{\d x} \left( \sum_{k=0}^\infty (\alpha_k \psi_k(x) + \beta_k (\psi_k * \rho)(x)) \rho(x) \right) + \sigma \frac{\d^2 \rho}{\d x^2}(x) = 0.
\end{equation}
Then, multiply the equation by the test function $\int \psi_i$, integrate over $\R$ and then by parts to obtain
\begin{equation}
\sum_{k=0}^\infty \alpha_k \int_\R \psi_i(x) \psi_k(x) \rho(x) \dd x + \sum_{k=0}^\infty \beta_k \int_\R \psi_i(x) (\psi_k * \rho)(x) \rho(x) \dd x = \sigma \int_\R \psi_i'(x) \rho(x) \dd x,
\end{equation}
which, due to the fact that the functions $\{ \psi_k \}_{k=0}^\infty$ are orthonormal, yields
\begin{equation} \label{eq:equation_infinity}
\alpha_i + \sum_{k=0}^\infty \beta_k \E^\rho[ \psi_i(X) (\psi_k * \rho)(X) ] = \sigma \E^\rho[\psi_i'(X)].
\end{equation}
Let us now introduce the notation
\begin{equation}
B_{ik} = \E^\rho[ \psi_i(X) (\psi_k * \rho)(X) ], \qquad \gamma_i = \E^\rho[\psi_i'(Y)],
\end{equation}
and notice that the quantities $\alpha_i$ and $\gamma_i$ can be computed as follows
\begin{equation} \label{eq:def_alpha_gamma}
\begin{aligned}
\alpha_i &= \int_\R V'(x) \psi_i(x) \rho(x) \dd x = \frac1{c_i} \sum_{j=0}^i \lambda_{ij} \int_\R V'(x) x^j \rho(x) \dd x = \frac1{c_i} \sum_{j=0}^i \lambda_{ij} \E^\rho \left[ V'(X) X^j \right], \\
\gamma_i &= \int_\R \psi_i'(x) \rho(x) \dd x = \frac1{c_i} \sum_{j=1}^i j \lambda_{ij} \int_\R x^{j-1} \rho(x) \dd x = \frac1{c_i} \sum_{j=1}^i j \lambda_{ij} \mathbb M^{(j-1)}.
\end{aligned}
\end{equation}
Moreover, by the binomial theorem, we have
\begin{equation} \label{eq:Bik}
\begin{aligned}
B_{ik} &= \int_\R \psi_i(x) \int_\R \psi_k(x-y) \rho(y) \dd y \rho(x) \dd x \\
&= \frac1{c_i c_k} \sum_{\ell=0}^i \sum_{j=0}^k \lambda_{i\ell} \lambda_{kj} \int_\R \int_\R x^\ell (x-y)^j \rho(y) \rho(x) \dd y \dd x \\
&= \frac1{c_i c_k} \sum_{\ell=0}^i \sum_{j=0}^k \lambda_{i\ell} \lambda_{kj} \int_\R \int_\R x^\ell \sum_{m=0}^j \binom{j}{m} (-1)^{j-m} x^m y^{j-m} \rho(y) \rho(x) \dd y \dd x \\
&= \frac1{c_i c_k} \sum_{\ell=0}^i \sum_{j=0}^k \sum_{m=0}^j (-1)^{j-m} \binom{j}{m} \lambda_{i\ell} \lambda_{kj} \mathbb M^{(m+\ell)} \mathbb M^{(j-m)}.
\end{aligned}
\end{equation}
If we then truncate the Fourier series up to order $K$ and therefore consider the indices $i,k = 0, \dots, K$, from equation \eqref{eq:equation_infinity}, we obtain the linear system of dimension $K+1$
\begin{equation} \label{eq:system_truncated}
B^{(K)} \beta^{(K)} = \sigma \gamma^{(K)} - \alpha^{(K)} - e^{(K)},
\end{equation}
where
\begin{equation} \label{eq:system_terms}
\begin{aligned}
\alpha^{(K)} &= \begin{bmatrix} \alpha_0 & \cdots & \alpha_K \end{bmatrix}^\top \in \R^{K+1}, \\
\beta^{(K)} &= \begin{bmatrix} \beta_0 & \cdots & \beta_K \end{bmatrix}^\top \in \R^{K+1}, \\
\gamma^{(K)} &= \begin{bmatrix} \gamma_0 & \cdots & \gamma_K \end{bmatrix}^\top \in \R^{K+1}, \\
e^{(K)} &= \begin{bmatrix} \epl_0^{(K)} & \cdots & \epl_K^{(K)} \end{bmatrix}^\top \in \R^{K+1} \quad \text{with} \quad \epl_i^{(K)} = \sum_{k=K+1}^\infty B_{ik} \beta_k,
\end{aligned}
\end{equation}
and $B \in \R^{(K+1) \times (K+1)}$ is the matrix whose entries are given in equation \eqref{eq:Bik}. It now remains to approximate the entries of the matrix and the right-hand side in the system \eqref{eq:system_truncated}, since they cannot be computed exactly. Due to ergodicity and propagation of chaos, using the empirical moments, and neglecting the reminder $e^{(K)}$, we define the linear system whose solution $\widetilde \beta^{(K)}_{T,N}$ is the estimator of the coefficients of the Fourier expansion of the interaction kernel as
\begin{equation} \label{eq:linear_system_approx}
\widetilde B^{(K)}_{T,N} \widetilde \beta^{(K)}_{T,N} = \sigma \widetilde \gamma^{(K)}_{T,N} - \widetilde \alpha^{(K)}_{T,N},
\end{equation}
where
\begin{equation} \label{eq:terms_linear_system_def}
\begin{aligned}
(\widetilde\alpha^{(K)}_{T,N})_i &= \frac1{\widetilde c_i} \sum_{j=0}^i \widetilde\lambda_{ij} \frac1T \int_0^T V'(Y_t) Y_t^j \dd t, \\
(\widetilde\gamma^{(K)}_{T,N})_i &= \frac1{\widetilde c_i} \sum_{j=1}^i j \widetilde\lambda_{ij} \widetilde{\mathbb M}^{(j-1)}, \\
(\widetilde B^{(K)}_{T,N})_{ik} &= \frac1{\widetilde c_i \widetilde c_k} \sum_{\ell=0}^i \sum_{j=0}^k \sum_{m=0}^j (-1)^{j-m} \binom{j}{m} \widetilde\lambda_{i\ell} \widetilde\lambda_{kj} \widetilde{\mathbb M}^{(m+\ell)} \widetilde{\mathbb M}^{(j-m)}.
\end{aligned}
\end{equation}
\begin{remark}
To solve the linear system \eqref{eq:linear_system_approx}, the diffusion coefficient $\sigma$ is assumed to be known, since this work focuses on the inference problem for the interaction kernel. Nevertheless, it is also possible to identify both the diffusion coefficient and the interaction kernel sequentially. Specifically, given a particle trajectory $(Y_t)_{t \in [0,T]}$ as defined in equation \eqref{eq:observations}, the diffusion coefficient can be estimated using the quadratic variation (see, e.g., \cite[Section 5.3]{Pav14})
\begin{equation}
\mathcal{Q}^\Delta_{T,N} = \sum_{j=1}^J (Y_{j\Delta} - Y_{(j-1)\Delta})^2,
\end{equation}
where $\Delta > 0$ is the time step and $J = T / \Delta$. In fact, it holds that
\begin{equation}
\lim_{\Delta \to 0} \mathcal{Q}^\Delta_{T,N} = 2\sigma T, \qquad \text{a.s.},
\end{equation}
independently of $N$, which implies that an estimator for the diffusion coefficient is given by
\begin{equation}
\widehat{\sigma}^\Delta_{T,N} = \frac{\mathcal{Q}^\Delta_{T,N}}{2T}.
\end{equation}
In addition, our spectral estimator from~\cite{PaZ22} can be used to infer the diffusion coefficient using low-frequency observations. The estimated diffusion coefficient can then be used in the subsequent estimation of the interaction kernel. A natural direction for future research is to rigorously analyze how the error in estimating the diffusion coefficient propagates to the kernel estimator. We expect this to lead to an additional error term in the convergence analysis, which would appear on the right-hand side of \cref{thm:estimate_W}.
\end{remark}
If $\det(\widetilde B^{(K)}_{T,N}) \neq 0$, then this definition is sufficient from the practical point of view, as we will observe in the following numerical experiments. However, in order to analyze the convergence properties of the estimator of the interaction kernel, similarly to what was done in \cite{PaZ24}, we need to introduce a sequence of convex and compact sets $\{ \mathcal B_K \}_{K=0}^\infty$, $\mathcal B_K \subset \R^{K+1}$, for the admissible coefficients. Then, let $\widehat \beta_{T,N}^{(K)}$ be the projection of the solution $\widetilde \beta^{(K)}_{T,N}$ of the linear system \eqref{eq:linear_system_approx} onto the convex and compact set $\mathcal B_K$
\begin{equation} \label{eq:estimator_beta}
\widehat \beta_{T,N}^{(K)} = \argmin_{\beta \in \mathcal B_K} \norm{\beta - \widetilde \beta^{(K)}_{T,N}}.
\end{equation}
Finally, the estimator of the interaction kernel is given by
\begin{equation} \label{eq:estimator_W}
(\widehat W')_{T,N}^{(K)}(x) = \sum_{k=0}^K (\widehat \beta_{T,N}^{(K)})_k \widetilde\psi_k(x) = \sum_{k=0}^K \frac{(\widehat \beta_{T,N}^{(K)})_k}{\widetilde c_k} \sum_{j=0}^k \widetilde\lambda_{kj} x^j,
\end{equation}
and in \cref{alg:procedure} we summarize the main steps needed for its construction. We emphasize again that the system size $N$ is not required to infer the interaction kernel. The subscript $N$ in the estimator $(\widehat{W}')_{T,N}^{(K)}$ simply indicates that it is based on observations generated by an interacting particle system with $N$ particles, and therefore it depends on $N$ only implicitly.

\begin{algorithm}[t]
\caption{Estimation of $W' \in L^2(\rho)$} \label{alg:procedure}
\begin{tabbing}
\hspace*{\algorithmicindent} \textbf{Input:} \= Drift function $V' \in L^2(\rho)$. \\
\> Observed trajectory $(Y_t)_{t\in[0,T]}$. \\
\> Number of Fourier coefficients $K$. \\
\> Set of admissible coefficients $\mathcal B_K \subset \R^{K+1}$. \\
\> Diffusion coefficient $\sigma > 0$.
\end{tabbing}
\begin{tabbing}
\hspace*{\algorithmicindent} \textbf{Output:} \= Estimator $(\widehat W')_{T,N}^{(K)}$ of $W'$.
\end{tabbing}
\begin{enumerate}[label=\arabic*:,itemsep=5pt]
\item Compute the approximated moments $\widetilde{\mathbb M}_{T,N}^{(r)}$ for $r = 0, \dots, 2K$ \\ from equation \eqref{eq:approximated_moments}.
\item Construct the approximated orthogonal polynomials $\{ \widetilde \psi_k \}_{k=0}^K$ \\ from equations \eqref{eq:ortho_poly_def}, \eqref{eq:lambda_def}, \eqref{eq:c_def}.
\item Construct the matrix $\widetilde B^{(K)}_{T,N} \in \R^{(K+1) \times (K+1)}$ and the vectors \\ $\widetilde \alpha^{(K)}_{T,N}, \widetilde \gamma^{(K)}_{T,N} \in \R^{K+1}$ from equation \eqref{eq:terms_linear_system_def}.
\item Compute the projection $\widehat \beta_{T,N}^K$ onto $\mathcal B_K$ of the solution $\widetilde \beta^{(K)}_{T,N}$ \\ of the linear system \eqref{eq:linear_system_approx}.
\item Construct the estimator $(\widehat W')_{T,N}^{(K)}$ from equation \eqref{eq:estimator_W}.
\end{enumerate}
\centering
\vspace{-0.25cm}
\noindent\rule{0.92\textwidth}{0.4pt}

\vspace{0.1cm}
\begin{tikzpicture}[scale=0.8]
\begin{small}
\draw[] (0,0) rectangle (3,2) node[pos=.5, align=center] {Empirical \\ moments \\ $\widetilde{\mathbb M}_{T,N}^{(r)}$};
\draw[->] (3,1) -- (4.5,1);
\draw[] (4.5,0) rectangle (7.5,2) node[pos=.5, align=center] {Orthonormal \\ basis \\ $\widetilde \psi_k$};
\draw[->] (7.5,1) -- (9,1);
\draw[] (9,0) rectangle (12,2) node[pos=.5, align=center] {Linear \\ system \\ $\widehat \beta^{(K)}_{T,N}$};
\draw[->] (12,1) -- (13.5,1);
\draw[] (13.5,0) rectangle (16.5,2) node[pos=.5, align=center] {Interaction \\ estimator \\ $(\widehat W')_{T,N}^{(K)}$};
\end{small}
\end{tikzpicture}
\end{algorithm}

\begin{remark}
We believe that one of the main advantages of the proposed inference methodology is that it requires the observation of a single-particle trajectory of the interacting particle system. On the other hand, if multiple trajectories are available, it would be possible to use them to obtain a more precise approximation of the interaction kernel. In particular, a more accurate estimate of the empirical moments could be obtained by averaging over all available particle trajectories. Alternatively, we could first compute an estimator of the Fourier coefficients of the interaction kernel for each particle and then average. We expect that averaging over many particle observations will reduce the variance of the resulting estimator. However, the study of the best averaging strategy and the variance of the estimator is beyond the scope of this work.
\end{remark}

Before studying the dependence of the estimator $(\widehat W')_{T,N}^{(K)}$ on its parameters $T,N,$ and $K$, in the next section we discuss the problem of inferring the drift term $V'$.

\subsection{Inference of the drift term}

In case also the drift term $V'$ needs to be estimated, and therefore the vector $\widetilde \alpha_{T,N}^{(K)}$ in the linear system \eqref{eq:linear_system_approx} cannot be computed, but it is an unknown, then we would have
\begin{equation}
\begin{bmatrix} \mathcal I_{K+1} & \widetilde B^{(K)}_{T,N} \end{bmatrix}
\begin{bmatrix} \widetilde \alpha^{(K)}_{T,N} \\ \widetilde \beta^{(K)}_{T,N} \end{bmatrix}
= \sigma \widetilde \gamma^{(K)}_{T,N},
\end{equation}
where $\mathcal I_{K+1} \in \R^{(K+1)\times(K+1)}$ denotes the identity matrix, and which is an underdetermined system. This can also be seen formally from equation \eqref{eq:equation_infinity}, where we have ``infinitely many'' linear equations for ``$2$ $\times$ infinitely many'' unknowns. Therefore, our methodology does not allow us to simultaneously infer both the drift term and the interaction kernel. We remark that this property is common to all the methods that rely on the observation of a single particle and consequently leverage the stationary Fokker--Planck equation to derive estimators. This is detailed in the next result.

\begin{proposition} \label{pro:identifiability}
Starting from the McKean--Vlasov PDE \eqref{eq:FP}, it is not possible to uniquely determine both the drift term and the interaction kernel in the interacting particle system \eqref{eq:interacting_particles} from the observation of a single trajectory.
\end{proposition}
\begin{proof}
We show that there exist infinite combinations of $V'$ and $W'$ that give the same stationary Fokker--Planck equation and therefore the same invariant measure. Let $f \colon \R \to \R$ and define $\widetilde V, \widetilde W \colon \R \to \R$ as
\begin{equation}
\widetilde V(x) = V(x) - (f*\rho)(x) \qquad \text{and} \qquad \widetilde W(x) = W(x) + f(x).
\end{equation}
Then, we have
\begin{equation}
\widetilde V' + (\widetilde W' * \rho) = V' - (f'*\rho) + (W'*\rho) + (f'*\rho) = V' + (W' * \rho),
\end{equation}
which implies that the stationary Fokker--Planck equation for $\widetilde V'$ and $\widetilde W'$ coincides with equation \eqref{eq:McKean_Vlasov} for $V'$ and $W'$.
\end{proof}

\begin{remark}
In principle, our methodology could recover both $V'$ and $W'$ if one further assumes that the subspaces spanned by their generalized Fourier coefficients are orthogonal, i.e., that the sets of non-zero Fourier coefficients with respect to the orthonormal basis are disjoint. However, this is not a realistic assumption in most practical settings
\end{remark}

Nevertheless, if one knows the interaction kernel $W'$ and is only interested in estimating the drift term $V'$, then the problem reduces to the inference problem of a scalar diffusion process and we have a closed-form expression for the Fourier coefficients of the drift term
\begin{equation}
\widetilde \alpha^{(K)}_{T,N} = \sigma \widetilde \gamma^{(K)}_{T,N} - \widetilde B^{(K)}_{T,N} \widetilde \beta^{(K)}_{T,N}.
\end{equation}
We also mention that in the absence of interaction between the particles, the estimator
\begin{equation}
\widetilde \alpha^{(K)}_{T,N} = \sigma \widetilde \gamma^{(K)}_{T,N}
\end{equation}
is not affected by the additional approximation error of ignoring $e^{(K)}$ in the exact system \eqref{eq:system_truncated}, since $e^{(K)} = 0$ in this case. Then, similarly to the previous section, the estimator of the drift term is given by
\begin{equation}
(\widehat V')_{T,N}^{(K)}(x) = \sum_{k=0}^K (\widehat \alpha_{T,N}^{(K)})_k \widetilde\psi_k(x) = \sum_{k=0}^K \frac{(\widehat \alpha_{T,N}^{(K)})_k}{\widetilde c_k} \sum_{j=0}^k \widetilde\lambda_{kj} x^j,
\end{equation}
where $\widehat \alpha_{T,N}^{(K)}$ is the projection of $\widetilde \alpha_{T,N}^{(K)}$ onto the set of admissible Fourier coefficients.

\subsection{Convergence analysis for the interaction kernel estimator} \label{sec:analysis_method}

In this section, we study the convergence properties of the estimator $(\widehat W')_{T,N}^{(K)}$ defined in equation \eqref{eq:estimator_W} as $T,N,K \to \infty$. We note that $K$ is the number of Fourier coefficients that we use to approximate the interaction kernel. Therefore, fixing $K$ gives the best possible approximation that can be reached in the ideal setting in which we observe an infinite trajectory from a system of infinite particles. On the other hand, we will observe that increasing $K$ results in a worse conditioning of the linear system to be solved, which in turn implies that larger values of $T$ and $N$ are necessary in order to obtain an accurate approximation of the interaction kernel. Our study follows the approach in \cite{PaZ24}, and is based on a forward error stability analysis for the linear system \eqref{eq:linear_system_approx}, which can be seen as a perturbation of equation \eqref{eq:system_truncated}. In the next lemma, whose proof is inspired by the proof of \cite[Lemma 4.5]{PaZ24}, we quantify the error that we commit in replacing the exact moments with the empirical moments on both the left-hand side and the right-hand side of the linear system.

\begin{lemma} \label{lem:estimate_terms_system}
Let $\gamma^{(K)}$, $\alpha^{(K)}$, $B^{(K)}$ be defined in equations \eqref{eq:def_alpha_gamma}, \eqref{eq:Bik}, \eqref{eq:system_truncated}, and let $\widetilde\gamma^{(K)}_{T,N}$, $\widetilde\alpha^{(K)}_{T,N}$, $\widetilde B^{(K)}_{T,N}$ be their approximations defined in equation \eqref{eq:terms_linear_system_def}. Under \cref{as:unique,as:inverse_ck,as:L2_rho}, for all $K \ge 0$ there exists a constant $C = C(K) > 0$, independent of $T$ and $N$, such that
\begin{equation}
\begin{aligned}
\mathrm{(i)}& \quad \E \left[ \norm{\widetilde \gamma^{(K)}_{T,N} - \gamma^{(K)}} \right] \le C(K) \left( \frac1{\sqrt T} + \frac1{\sqrt N} \right), \\
\mathrm{(ii)}& \quad \E \left[ \norm{\widetilde \alpha^{(K)}_{T,N} - \alpha^{(K)}} \right] \le C(K) \left( \frac1{\sqrt T} + \frac1{\sqrt N} \right), \\
\mathrm{(iii)}& \quad \E \left[ \norm{\widetilde B^{(K)}_{T,N} - B^{(K)}} \right] \le C(K) \left( \frac1{\sqrt T} + \frac1{\sqrt N} \right).
\end{aligned}
\end{equation}
\end{lemma}
\begin{proof}
Let us start from (i). By definition of $\widetilde \gamma^{(K)}_{T,N}$ and $\gamma^{(K)}$, we have
\begin{equation}
\begin{aligned}
\E \left[ \norm{\widetilde \gamma^{(K)}_{T,N} - \gamma^{(K)}} \right] &= \E \left[ \sqrt{ \sum_{i=0}^K \left( \frac1{\widetilde c_i} \sum_{j=1}^i j \widetilde\lambda_{ij} \widetilde{\mathbb M}^{(j-1)} - \frac1{c_i} \sum_{j=1}^i j \lambda_{ij} \mathbb M^{(j-1)} \right)^2 } \right] \\
&\le \sum_{i=0}^K \E \left[ \abs{ \frac1{\widetilde c_i} \sum_{j=1}^i j \widetilde\lambda_{ij} \widetilde{\mathbb M}^{(j-1)} - \frac1{c_i} \sum_{j=1}^i j \lambda_{ij} \mathbb M^{(j-1)} } \right],
\end{aligned}
\end{equation}
and using the triangle inequality we get
\begin{equation} \label{eq:decomposition_gamma}
\begin{aligned}
\E \left[ \norm{\widetilde \gamma^{(K)}_{T,N} - \gamma^{(K)}} \right] &\le \sum_{i=0}^K \sum_{j=1}^i j \left( \E \left[ \frac1{\widetilde c_i} \abs{\widetilde\lambda_{ij} - \lambda_{ij}} \abs{\widetilde{\mathbb M}^{(j-1)}} \right] + \abs{\lambda_{ij}} \E \left[ \frac1{\widetilde c_i} \abs{\widetilde{\mathbb M}^{(j-1)} - \mathbb M^{(j-1)}} \right] \right) \\
&\quad + \sum_{i=0}^K \E \left[ \abs{\frac1{\widetilde c_i} - \frac1{c_i}} \right] \left( \sum_{j=1}^i j \abs{\lambda_{ij} \mathbb M^{(j-1)}} \right).
\end{aligned}
\end{equation}
Then, using Hölder's inequality we obtain for $r$ given by \cref{as:inverse_ck}
\begin{equation} \label{eq:bound_gamma_1}
\E \left[ \frac1{\widetilde c_i} \abs{\widetilde{\mathbb M}^{(j-1)} - \mathbb M^{(j-1)}} \right] \le \left( \E \left[ \frac1{\widetilde c_i^r} \right] \right)^{\frac1{r}} \left( \E \left[ \abs{\widetilde{\mathbb M}^{(j-1)} - \mathbb M^{(j-1)}}^{\frac{r}{r-1}} \right] \right)^{\frac{r-1}{r}},
\end{equation}
and
\begin{equation} \label{eq:bound_gamma_2}
\begin{aligned}
\E \left[ \frac1{\widetilde c_i} \abs{\widetilde\lambda_{ij} - \lambda_{ij}} \abs{\widetilde{\mathbb M}^{(j-1)}} \right] &\le \left( \E \left[ \frac1{\widetilde c_i^r} \right] \right)^{\frac1{r}} \left( \E \left[ \abs{\widetilde\lambda_{ij} - \lambda_{ij}}^{\frac{3r-2}{2(r-1)}} \right] \right)^{\frac{2(r-1)}{3r-2}} \\
&\qquad \times \left( \E \left[ \abs{\widetilde{\mathbb M}^{(j-1)}}^{\frac{r(3r-2)}{(r-2)(r-1)}} \right] \right)^{\frac{(r-2)(r-1)}{r(3r-2)}},
\end{aligned}
\end{equation}
where we notice that
\begin{equation}
\frac1{r} + \frac{2(r-1)}{3r-2} + \frac{(r-2)(r-1)}{r(3r-2)} = 1, \qquad \frac{3r-2}{2(r-1)} \in [1,2), \qquad \frac{r(3r-2)}{(r-2)(r-1)} \ge 1.
\end{equation}
Therefore, using bounds \eqref{eq:bound_gamma_1} and \eqref{eq:bound_gamma_2} in equation \eqref{eq:decomposition_gamma}, due to the boundedness of the moments, equation \eqref{eq:estimate_moments}, \cref{as:inverse_ck}, and \cref{lem:estimate_lambda,lem:estimate_inverse_ck}, we deduce (i). We proceed similarly for (ii) and we get
\begin{equation} \label{eq:decomposition_alpha}
\begin{aligned}
\E \left[ \norm{\widetilde \alpha^{(K)}_{T,N} - \alpha^{(K)}} \right] &\le \sum_{i=0}^K \sum_{j=0}^i \E \left[ \frac1{\widetilde c_i} \abs{\widetilde\lambda_{ij} - \lambda_{ij}} \abs{\frac1T \int_0^T V'(Y_t) Y_t^j \dd t} \right] \\
&\quad + \sum_{i=0}^K \sum_{j=0}^i \abs{\lambda_{ij}} \E \left[ \frac1{\widetilde c_i} \abs{\frac1T \int_0^T V'(Y_t) Y_t^j \dd t - \E^\rho \left[ V'(X) X^j \right]} \right] \\
&\quad + \sum_{i=0}^K \E \left[ \abs{\frac1{\widetilde c_i} - \frac1{c_i}} \right] \left( \sum_{j=0}^i \abs{\lambda_{ij} \E^\rho \left[ V'(X) X^j \right]} \right).
\end{aligned}
\end{equation}
The next steps are still analogous to the ones for (i), and therefore (ii) follows if we show that for all $p \ge 1$ and for all $q \in [1,2)$ there exists a constant $\widetilde C > 0$, independent of $T$ and $N$, such that
\begin{align}
\left( \E \left[ \abs{\frac1T \int_0^T V'(Y_t) Y_t^j \dd t}^p \right] \right)^{1/p} &\le \widetilde C, \label{eq:result_1} \\
E \defeq \left( \E \left[ \abs{\frac1T \int_0^T V'(Y_t) Y_t^j \dd t - \E^\rho \left[ V'(X) X^j \right]}^q \right] \right)^{1/q} &\le \widetilde C \left( \frac1{\sqrt T} + \frac1{\sqrt N} \right). \label{eq:result_2}
\end{align}
First, by the Hölder's inequality, we have
\begin{equation}
\E \left[ \abs{\frac1T \int_0^T V'(Y_t) Y_t^j \dd t}^p \right] \le \frac1T \int_0^T \E \left[ \abs{V'(Y_t) Y_t^j}^p \right] \dd t,
\end{equation}
and the bound \eqref{eq:result_1} follows since $V'$ is polynomially bounded and due to the boundedness of the moments. For equation \eqref{eq:result_2}, we proceed as in the proof of \cite[Lemma 4.3]{PaZ24} and we obtain
\begin{equation}
\begin{aligned}
E &\le \left( \E \left[ \abs{\frac1T \int_0^T V'(X_t) X_t^j \dd t - \E^\rho \left[ V'(X) X^j \right]}^2 \right] \right)^{1/2} \\
&\quad + \left( \frac1T \int_0^T \E \left[ \abs{V'(Y_t)- V'(X_t)}^q \abs{Y_t^j}^q \right] \dd t \right)^{1/q} \\
&\quad + \left( \frac1T \int_0^T \E \left[ \abs{V'(X_t)}^q \abs{Y_t^j - X_t^j}^q \right] \dd t \right)^{1/q}, \\
&\eqdef E_1 + E_2 + E_3.
\end{aligned}
\end{equation}
Then, for $E_1$ and $E_3$, still following the proof of \cite[Lemma 4.3]{PaZ24}, we apply the mean ergodic theorem in \cite[Section 4]{MST10}, Hölder's and Jensen's inequality, the boundedness of the moments, and the uniform propagation of chaos property to get $E_1 \le \widetilde C/\sqrt T$ and $E_3 \le \widetilde C / \sqrt N$. Regarding $E_2$, using the mean value theorem, we have
\begin{equation}
V'(Y_t) - V'(X_t) = V''(Z_t) (Y_t - X_t),
\end{equation} 
where $Z_t$ takes values between $Y_t$ and $X_t$, and applying again Hölder's and Jensen's inequality, the boundedness of the moments, and the uniform propagation of chaos property we deduce $E_2 \le \widetilde C / \sqrt N$, which yields (ii). Let us now consider (iii). Since the spectral norm is bounded by the Frobenius norm, we have
\begin{equation}
\begin{aligned}
\E \left[ \norm{\widetilde B^{(K)}_{T,N} - B^{(K)}} \right] &\le \sum_{i=0}^K \sum_{k=0}^K \E \left[ \left| \frac1{\widetilde c_i \widetilde c_k} \sum_{\ell=0}^i \sum_{j=0}^k \sum_{m=0}^j (-1)^{j-m} \binom{j}{m} \widetilde\lambda_{i\ell} \widetilde\lambda_{kj} \widetilde{\mathbb M}^{(m+\ell)} \widetilde{\mathbb M}^{(j-m)} \right. \right. \\
&\hspace{2cm} \left. \left. - \frac1{c_i c_k} \sum_{\ell=0}^i \sum_{j=0}^k \sum_{m=0}^j (-1)^{j-m} \binom{j}{m} \lambda_{i\ell} \lambda_{kj} \mathbb M^{(m+\ell)} \mathbb M^{(j-m)} \right| \right] \\
&\le \sum_{i=0}^K \sum_{k=0}^K \sum_{\ell=0}^i \sum_{j=0}^k \sum_{m=0}^j \binom{j}{m} E_{ikljm},
\end{aligned}
\end{equation}
where
\begin{equation}
E_{ikljm} = \E \left[ \abs{\frac1{\widetilde c_i \widetilde c_k} \widetilde\lambda_{i\ell} \widetilde\lambda_{kj} \widetilde{\mathbb M}^{(m+\ell)} \widetilde{\mathbb M}^{(j-m)} - \frac1{c_i c_k} \lambda_{i\ell} \lambda_{kj} \mathbb M^{(m+\ell)} \mathbb M^{(j-m)}} \right].
\end{equation}
Then, applying multiple triangle inequalities to isolate the differences of all the terms appearing in the right-hand side, we obtain
\begin{equation}
\begin{aligned}
E_{ikljm} &\le \E \left[ \frac1{\widetilde c_i \widetilde c_k} \abs{\widetilde\lambda_{i\ell}} \abs{\widetilde\lambda_{kj}} \abs{\widetilde{\mathbb M}^{(m+\ell)}} \abs{\widetilde{\mathbb M}^{(j-m)} - \mathbb M^{(j-m)}} \right] \\
&\quad + \abs{\mathbb M^{(j-m)}} \E \left[ \frac1{\widetilde c_i \widetilde c_k} \abs{\widetilde\lambda_{i\ell}} \abs{\widetilde\lambda_{kj}} \abs{\widetilde{\mathbb M}^{(m+\ell)} - \mathbb M^{(m+\ell)}} \right] \\
&\quad + \abs{\mathbb M^{(m+\ell)}} \abs{\mathbb M^{(j-m)}} \E \left[ \frac1{\widetilde c_i \widetilde c_k} \abs{\widetilde\lambda_{i\ell}} \abs{\widetilde\lambda_{kj} - \lambda_{kj}} \right] \\
&\quad + \abs{\lambda_{kj}} \abs{\mathbb M^{(m+\ell)}} \abs{\mathbb M^{(j-m)}} \E \left[ \frac1{\widetilde c_i \widetilde c_k} \abs{\widetilde\lambda_{i\ell} - \lambda_{i\ell}} \right] \\
&\quad + \abs{\lambda_{i\ell}} \abs{\lambda_{kj}} \abs{\mathbb M^{(m+\ell)}} \abs{\mathbb M^{(j-m)}} \E \left[ \frac1{\widetilde c_i} \abs{\frac1{\widetilde c_k} - \frac1{c_k}} \right] \\
&\quad + \frac1{c_k} \abs{\lambda_{i\ell}} \abs{\lambda_{kj}} \abs{\mathbb M^{(m+\ell)}} \abs{\mathbb M^{(j-m)}} \E \left[ \abs{\frac1{\widetilde c_i} - \frac1{c_i}} \right].
\end{aligned}
\end{equation}
Finally, we estimate all the rows in the right-hand side applying Hölder's inequality to the expectations with exponents 
\begin{center}
\begin{tabular}{llllll}
$p_1 = r$, & $p_2 = r$, & $p_3 = \frac{3r(3r-4)}{(r-2)(r-4)}$, & $p_4 = \frac{3r(3r-4)}{(r-2)(r-4)}$, & $p_5 = \frac{3r(3r-4)}{(r-2)(r-4)}$, & $p_6 = \frac{3r-4}{2r-4}$, \\
$p_1 = r$, & $p_2 = r$, & $p_3 = \frac{2r(3r-4)}{(r-2)(r-4)}$, & $p_4 = \frac{2r(3r-4)}{(r-2)(r-4)}$, & $p_5 = \frac{3r-4}{2r-4}$, & \\
$p_1 = r$, & $p_2 = r$, & $p_3 = \frac{r(3r-4)}{(r-2)(r-4)}$, & $p_4 = \frac{3r-4}{2r-4}$, & & \\
$p_1 = r$, & $p_2 = r$, & $p_3 = \frac{r}{r-2}$, & & & \\
$p_1 = r$, & $p_2 = \frac{r}{r-1}$, & & & & \\
$p_1 = 1$, & & & & &
\end{tabular}
\end{center}
respectively, and then using the boundedness of the moments, equation \eqref{eq:estimate_moments}, \cref{as:inverse_ck}, and \cref{lem:estimate_lambda,lem:estimate_inverse_ck}, we obtain (iii), which concludes the proof.
\end{proof}

The second approximation we make by passing from the exact system \eqref{eq:system_truncated} to the final system \eqref{eq:linear_system_approx} is to remove the term $e^{(K)}$. In the next result, we justify this step by showing that $e^{(K)}$ vanishes for large values of $K$.

\begin{lemma} \label{lem:estimate_eK}
Let $e^{(K)}$ be defined in equation \eqref{eq:system_terms}. Under \cref{as:unique,as:L2_rho}, it holds
\begin{equation}
\lim_{K \to \infty} \norm{e^{(K)}} = 0.
\end{equation}
\end{lemma}
\begin{proof}
Let us first notice that since $W'*\rho \in L^2(\rho)$, then we can write
\begin{equation}
(W'*\rho)(x) = \sum_{i=0}^\infty \theta_i \psi_i(x),
\end{equation}
where
\begin{equation} \label{eq:theta_i}
\theta_i = \int_\R \psi_i(x) (W'*\rho)(x) \rho(x) \dd x = \sum_{k=0}^\infty \beta_k \int_\R \psi_i(x) (\psi_k * \rho)(x) \rho(x) \dd x = \sum_{k=0}^\infty B_{ik} \beta_k.
\end{equation}
Moreover, by Parseval's theorem, we have
\begin{equation} \label{eq:Parseval}
\norm{W'*\rho}_{L^2(\rho)}^2 = \sum_{i=0}^\infty \theta_i^2 = \sum_{i=0}^\infty \left( \sum_{k=0}^\infty B_{ik} \beta_k \right)^2.
\end{equation}
Then, by definition of $e^{(K)}$ we get
\begin{equation}
\norm{e^{(K)}}^2 = \sum_{i=0}^K \left( \sum_{k=K+1}^\infty B_{ik} \beta_k \right)^2 = \sum_{i=0}^K \left( \sum_{k=0}^\infty B_{ik} \beta_k - \sum_{k=0}^K B_{ik} \beta_k \right)^2,
\end{equation}
which implies
\begin{equation}
\norm{e^{(K)}}^2 = \sum_{i=0}^K \left( \sum_{k=0}^\infty B_{ik} \beta_k \right)^2 + \sum_{i=0}^K \left( \sum_{k=0}^K B_{ik} \beta_k \right)^2 - 2 \sum_{i=0}^K \left( \sum_{k=0}^\infty B_{ik} \beta_k \right) \left( \sum_{k=0}^K B_{ik} \beta_k \right).
\end{equation}
Therefore, taking the limit as $K \to \infty$ and using equation \eqref{eq:Parseval} we obtain the desired result.
\end{proof}

We can now estimate the error of the estimator $\widehat \beta_{T,N}^{(K)}$ in equation \eqref{eq:estimator_beta} with respect to the true Fourier coefficients $\beta^{(K)}$. The proof of the next result is based on \cite[Section 3.1.2]{QSS98} and \cite[Theorem 4.6]{PaZ24}.

\begin{proposition} \label{pro:estimate_beta}
Let $\beta^{(K)}$ be defined in equation \eqref{eq:system_terms}, and let $\widehat \beta_{T,N}^{(K)}$ be the estimator given in equation \eqref{eq:estimator_beta}. Under \cref{as:unique,as:inverse_ck,as:L2_rho}, for all $K \ge 0$, if $\det(B^{(K)}) \neq 0$ where $B^{(K)}$ is defined in equations \eqref{eq:Bik}, \eqref{eq:system_truncated}, then there exists a constant $C = C(K) > 0$, independent of $T$ and $N$, such that
\begin{equation}
\E \left[ \norm{\widehat \beta_{T,N}^{(K)} - \beta^{(K)}} \right] \le C(K) \left( \frac1{\sqrt T} + \frac1{\sqrt N} \right) + \delta(K),
\end{equation}
where
\begin{equation}
\delta(K) = 2\norm{(B^{(K)})^{-1} e^{(K)}}.
\end{equation}
\end{proposition}
\begin{proof}
First, define the event $A_K$ as
\begin{equation}
A_K = \left\{ \norm{(B^{(K)})^{-1}} \norm{\widetilde B_{T,N}^{(K)} - B^{(K)}} < \frac12 \right\},
\end{equation}
and due to Markov's inequality and \cref{lem:estimate_terms_system}(iii) we have
\begin{equation} \label{eq:prob_compl}
\Pr(A_K^\compl) \le 2 \norm{(B^{(K)})^{-1}} \E \left[ \norm{\widetilde B_{T,N}^{(K)} - B^{(K)}} \right] \le C(K) \left( \frac1{\sqrt T} + \frac1{\sqrt N} \right).
\end{equation}
Therefore, using the law of total expectation and the fact that $\widehat \beta_{T,N}^{(K)}$ is the projection of $\widetilde \beta_{T,N}^{(K)}$ onto the convex and compact set $\mathcal B_K$ we obtain
\begin{equation} \label{eq:total_expectation}
\begin{aligned}
\E \left[ \norm{\widehat \beta_{T,N}^{(K)} - \beta^{(K)}} \right] &= \E \left[ \norm{\widehat \beta_{T,N}^{(K)} - \beta^{(K)}} | A_K \right] \Pr(A_K) + \E \left[ \norm{\widehat \beta_{T,N}^{(K)} - \beta^{(K)}} A_K^\compl \right] \Pr(A_K^\compl) \\
&\le \E \left[ \norm{\widetilde \beta_{T,N}^{(K)} - \beta^{(K)}} | A_K \right] + C(K) \left( \frac1{\sqrt T} + \frac1{\sqrt N} \right).
\end{aligned}
\end{equation}
Then, using equations \eqref{eq:system_truncated} and \eqref{eq:linear_system_approx} we can write
\begin{equation}
\begin{aligned}
\widetilde \beta_{T,N}^{(K)} - \beta^{(K)} &= \left( \mathcal I_{K+1} + (B^{(K)})^{-1}(\widetilde B_{T,N}^{(K)} - B^{(K)}) \right)^{-1} \\
&\qquad \times \left[ (B^{(K)})^{-1} \left( \sigma (\widetilde \gamma_{T,N}^{(K)} - \gamma^{(K)}) - (\widetilde \alpha_{T,N}^{(K)} - \alpha^{(K)}) \right) + (B^{(K)})^{-1} e^{(K)} \right] \\
&\quad - \left( \mathcal I_{K+1} + (B^{(K)})^{-1}(\widetilde B_{T,N}^{(K)} - B^{(K)}) \right)^{-1} (B^{(K)})^{-1} (\widetilde B_{T,N}^{(K)} - B^{(K)}) \beta^{(K)},
\end{aligned}
\end{equation}
which, following the proof of \cite[Theorem 3.1]{QSS98}, implies
\begin{equation}
\begin{aligned}
\E \left[ \norm{\widetilde \beta_{T,N}^{(K)} - \beta^{(K)}} | A_K \right] &\le 2 \norm{(B^{(K)})^{-1}} \E \left[ \sigma \norm{\widetilde\gamma_{T,N}^{(K)} - \gamma^{(K)}} + \norm{\widetilde\alpha_{T,N}^{(K)} - \alpha^{(K)}} | A_K \right] \\
&\hspace{-0.2cm} + 2 \norm{(B^{(K)})^{-1} e^{(K)}} + 2 \norm{(B^{(K)})^{-1}} \norm{\beta^{(K)}} \E \left[ \norm{\widetilde B_{T,N}^{(K)} - B^{(K)}} | A_K \right].
\end{aligned}
\end{equation}
Using the fact that $\E[Z|A_K] \le \E[Z]/\Pr(A_K)$ for a positive random variable $Z$, applying \cref{lem:estimate_terms_system}, and due to equation \eqref{eq:prob_compl}, we get for $T$ and $N$ sufficiently large
\begin{equation}
\E \left[ \norm{\widetilde \beta_{T,N}^{(K)} - \beta^{(K)}} | A_K \right] \le C(K) \left( \frac1{\sqrt T} + \frac1{\sqrt N} \right) + 2 \norm{(B^{(K)})^{-1} e^{(K)}},
\end{equation}
which, together with equation \eqref{eq:total_expectation}, gives the desired result.
\end{proof}

Finally, in the next result we consider the estimator $(\widehat W')_{T,N}^{(K)}$ in equation \eqref{eq:estimator_W}, and, employing \cref{pro:estimate_beta}, we analyze the approximation error with respect to the true interaction kernel $W'$.

\begin{theorem} \label{thm:estimate_W}
Let $(\widehat W')_{T,N}^{(K)}$ be the interaction kernel estimator defined in equation \eqref{eq:estimator_W}. Under \cref{as:unique,as:inverse_ck,as:L2_rho}, for all $K \ge 0$, if $\det(B^{(K)}) \neq 0$ where $B^{(K)}$ is defined in equations \eqref{eq:Bik}, \eqref{eq:system_truncated}, then there exists a constant $C = C(K) > 0$, independent of $T$ and $N$, such that
\begin{equation}
\E \left[ \norm{(\widehat W')_{T,N}^{(K)} - W'}_{L^2(\rho)} \right] \le C(K) \left( \frac1{\sqrt T} + \frac1{\sqrt N} \right) + \delta(K) + \epsilon(K),
\end{equation}
where
\begin{equation}
\delta(K) = 2\norm{(B^{(K)})^{-1} e^{(K)}} \qquad \text{and} \qquad \epsilon(K) = \sqrt{\sum_{k=K+1}^\infty \beta_k^2}.
\end{equation}
\end{theorem}
\begin{proof}
By definition of $(\widehat W')_{T,N}^{(K)}$ and considering the Fourier expansion of $W'$, we have
\begin{equation}
\norm{(\widehat W')_{T,N}^{(K)} - W'}_{L^2(\rho)} = \norm{\sum_{k=0}^K (\widehat \beta_{T,N}^{(K)})_k \widetilde\psi_k - \sum_{k=0}^\infty \beta_k \psi_k}_{L^2(\rho)},
\end{equation}
which, due to the triangle inequality, implies
\begin{equation}
\begin{aligned}
\norm{(\widehat W')_{T,N}^{(K)} - W'}_{L^2(\rho)} &\le \norm{\sum_{k=0}^K (\widehat \beta_{T,N}^{(K)})_k (\widetilde\psi_k - \psi_k)}_{L^2(\rho)} \\
&\quad + \norm{\sum_{k=0}^K ((\widehat \beta_{T,N}^{(K)})_k - (\beta^{(K)})_k \psi_k}_{L^2(\rho)} + \norm{\sum_{k=K+1}^\infty \beta_k \psi_k}_{L^2(\rho)}.
\end{aligned}
\end{equation}
Then, since $\widehat \beta_{T,N}^{(K)}$ belongs to the compact set $\mathcal B_K$ and $\{ \psi_k \}_{k=0}^\infty$ is orthonormal in $L^2(\rho)$, we have
\begin{equation}
\begin{aligned}
\E \left[ \norm{(\widehat W')_{T,N}^{(K)} - W'}_{L^2(\rho)} \right] &\le C(K) \sum_{k=0}^K \E \left[ \norm{\widetilde\psi_k - \psi_k}_{L^2(\rho)} \right] + \E \left[ \norm{\widehat \beta_{T,N}^{(K)} - \beta} \right] + \sqrt{\sum_{k=K+1}^\infty \beta_k^2},
\end{aligned}
\end{equation}
which, due to \cref{pro:estimate_psi,pro:estimate_beta}, gives the desired result.
\end{proof}

\begin{remark}
In order for the estimator to converge, the additional terms $\delta(K)$ and $\epsilon(K)$ in \cref{thm:estimate_W} must disappear as $K$ increases. First, notice that $\epsilon(K) \to 0$ as $K \to \infty$ because $W' \in L^2(\rho)$. In principle, we should be able to relate the decay of the generalized Fourier coefficients of the interaction kernel (or its derivative) with the regularity of $W$, using results from approximation theory such as Jackson’s inequality, as was done in~\cite{GHR04}. In our setting, we would need to study this problem in weighted Sobolev spaces on the whole real line, but the study of this interesting question is beyond the scope of this paper. Regarding $\delta(K)$, since, by \cref{lem:estimate_eK}, $\|e^{(K)}\|$ vanishes, it suffices to require that $\|(B^{(K)})^{-1}\|$ do not blow up faster than the rate at which $\|e^{(K)}\|$ converges to zero. The proof of this fact is nontrivial, as it depends on the unknown interaction kernel $W'$. However, we observed that this was indeed the case in all of the numerical experiments that we considered. In particular, this is straightforward to verify for all polynomial interactions, since $e^{(K)} = 0$ (and also $\epsilon(K) = 0$) for all $K \ge r$ with $r$ being the degree of the polynomial, which in turn implies $\delta(K) = 0$. Alternatively, notice that we can write 
\begin{equation}
(B^{(K)})^{-1} e^{(K)} = (B^{(K)})^{-1} \left( \theta^{(K)} - B^{(K)} \beta^{(K)} \right) = (B^{(K)})^{-1} \theta^{(K)} - \beta^{(K)},
\end{equation}
where $\theta^{(K)} \in \R^{K+1}$ is the vector of the Fourier coefficients of $W'*\rho$, whose components $\theta_i$, $i = 0, \dots, K$, are given in equation \eqref{eq:theta_i}. Therefore, requiring that $\delta(K)$ vanishes is equivalent to requiring that the solution $b^{(K)}$ of the linear system $B^{(K)} b^{(K)} = \theta^{(K)}$ converges to $\beta^{(K)}$ as $K \to \infty$, which is reasonable to assume since we already know that $\theta_i = \sum_{k=0}^\infty B_{ik} \beta_k$ still by equation \eqref{eq:theta_i}.
\end{remark}

\section{Numerical experiments} \label{sec:numerics}

In this section, we employ the methodology introduced above to infer interaction kernels in particle systems, and verify numerically the estimates predicted by the theory. We first consider the mean-field Ornstein--Uhlenbeck process, for which a basis of orthogonal polynomials with respect to the invariant measure can be computed analytically. Even though the Ornstein--Uhlenbeck process is a simple test case, it is still an interesting example because it allows us to assess the performance of our method, since both the invariant measure and the corresponding orthogonal polynomials are given in closed form. Then, we consider more complex interaction kernels, including examples where the assumptions of our theoretical analysis are not necessarily satisfied. Even in this case, numerical experiments demonstrate that our methodology can still be used to learn the interaction kernel from a single trajectory. We generate synthetic observations by numerically solving the interacting particle system \eqref{eq:interacting_particles} with deterministic initial conditions, $X_0^{(n)} = 0$ for all $n = 1, \dots, N$. The SDE system is discretized employing the Euler--Maruyama scheme with a time step $h = 0.01$. Then, we assume to observe only the first particle in the system to infer the interaction kernel, so that $Y_t = X_t^{(1)}$ for all $t \in [0,T]$. We notice that, while $N$ is used in the simulation to generate the trajectories, it is not needed for the estimation once the observations are available. In fact, our approach assumes access to a single observed trajectory, with no additional information about the system size.

\subsection{Mean-field Ornstein--Uhlenbeck process}

\begin{figure}[t]
\begin{center}
$T = 10\,000$ \\
\includegraphics{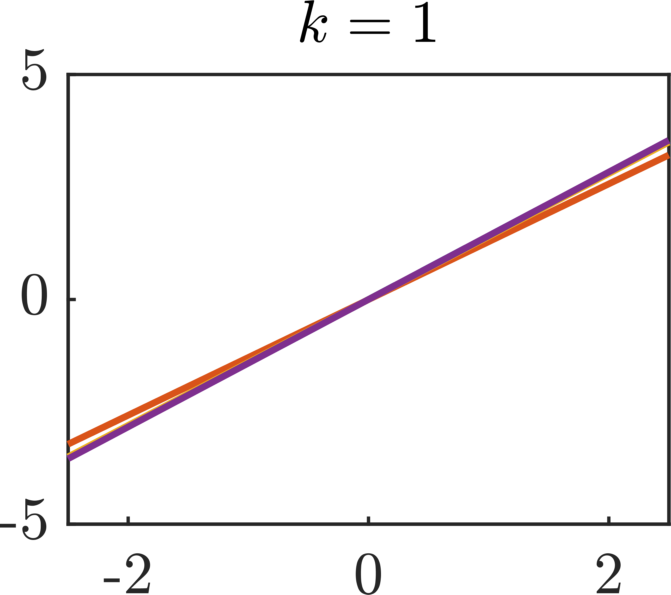}
\includegraphics{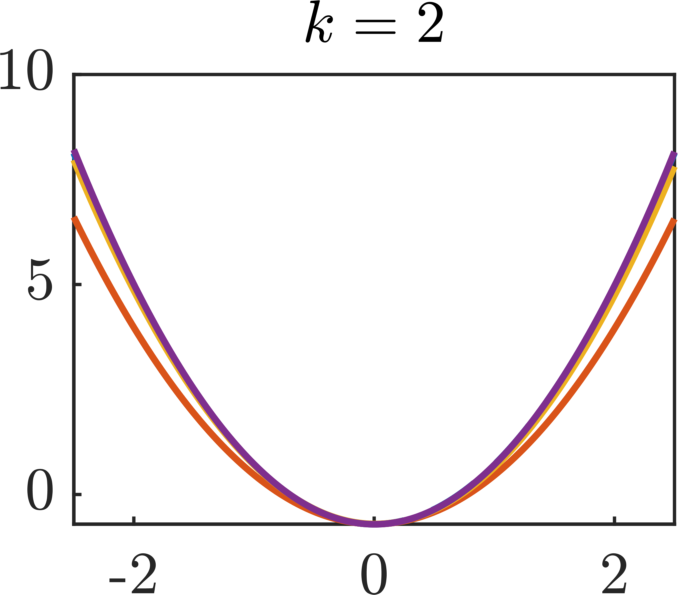}
\includegraphics{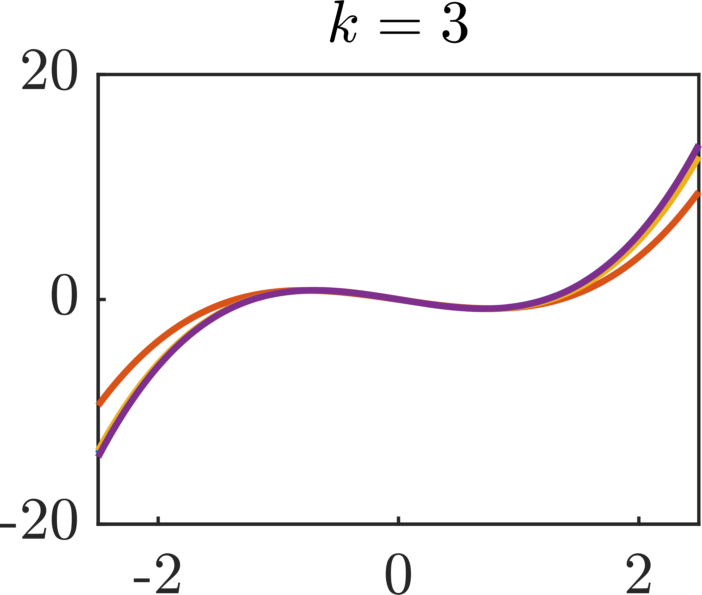}
\includegraphics{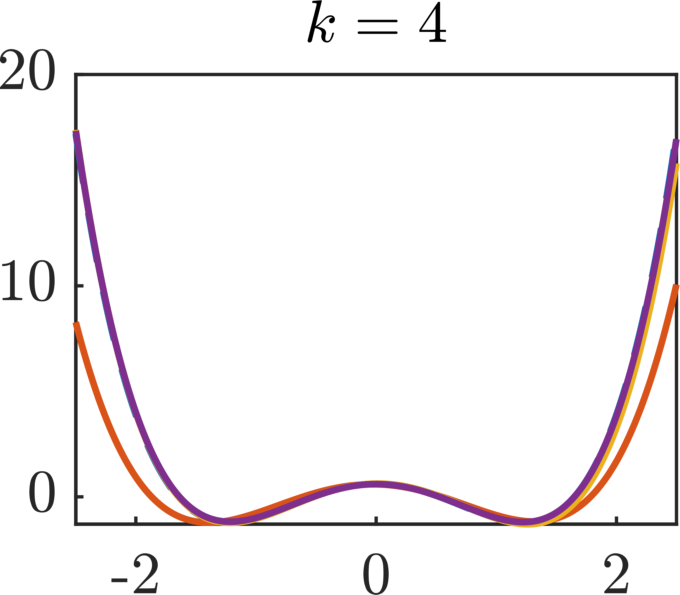} \\
\vspace{0.2cm}
\includegraphics{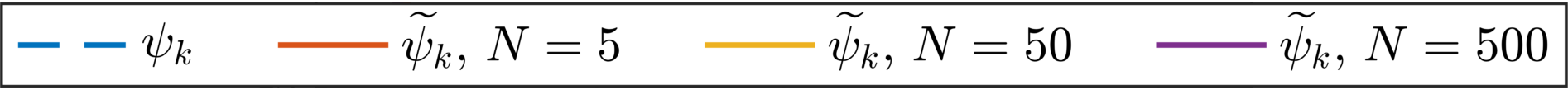} \\
\vspace{0.5cm}
$N = 500$ \\
\includegraphics{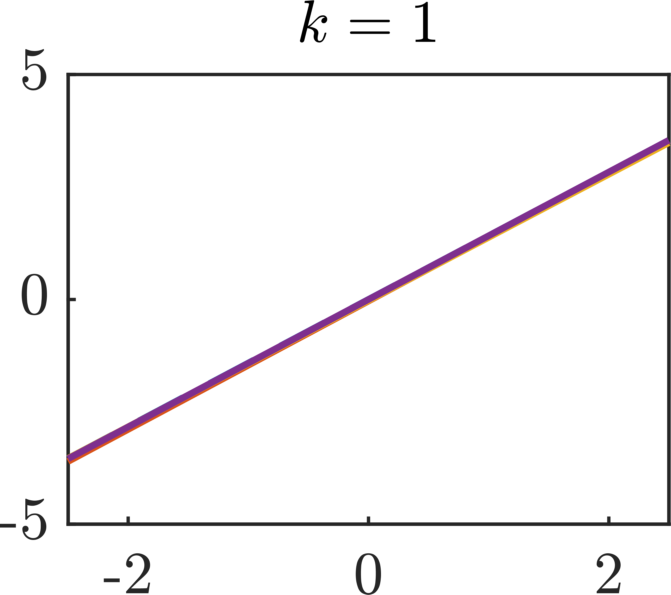}
\includegraphics{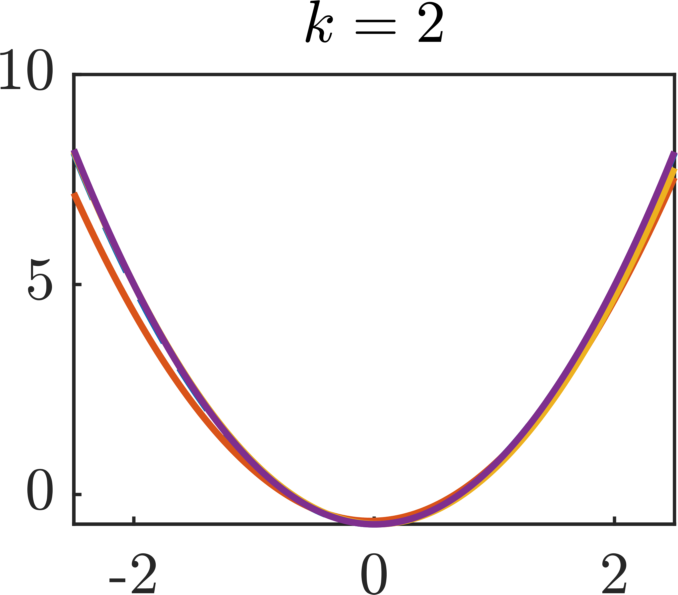}
\includegraphics{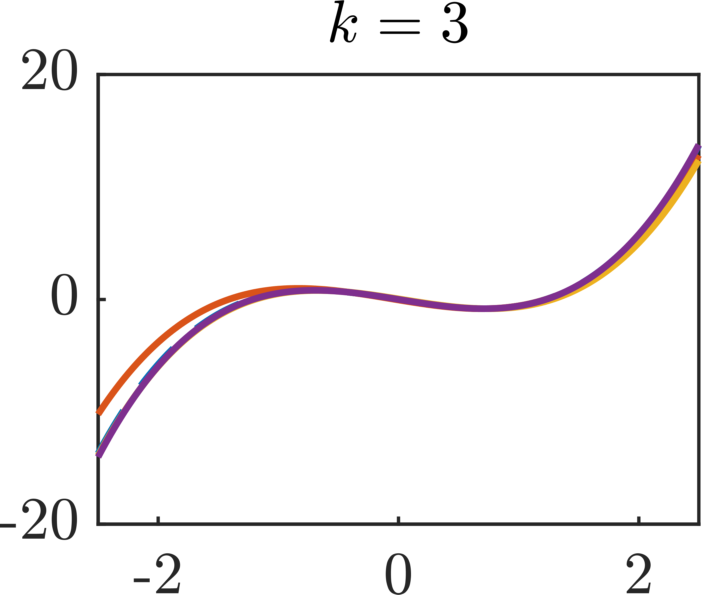}
\includegraphics{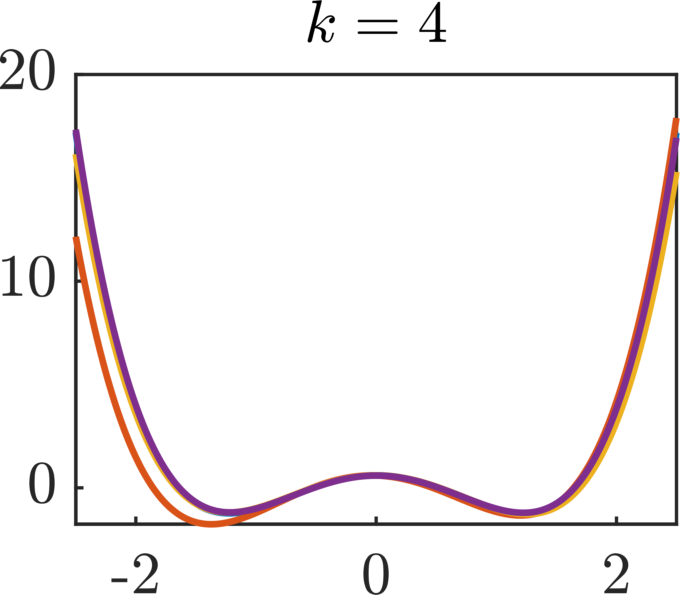} \\
\vspace{0.2cm}
\includegraphics{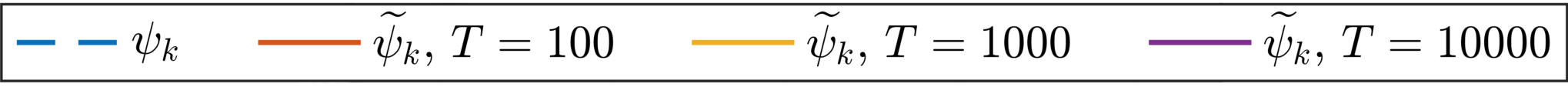}
\end{center}
\caption{Comparison between the first four (excluding the constant function) exact ($\psi_k$) and approximated ($\widetilde\psi_k$) orthogonal polynomials with respect to the invariant measure $\rho$ of the mean-field Ornstein--Uhlenbeck process. Top: we fix $T = 10\,000$ and vary $N = 5, 50, 500$. Bottom: we fix $N = 500$ and vary $T = 100, 1\,000, 10\,000$.}
\label{fig:orthogonal_polynomials}
\end{figure}

\begin{figure}[t]
\begin{center}
\includegraphics{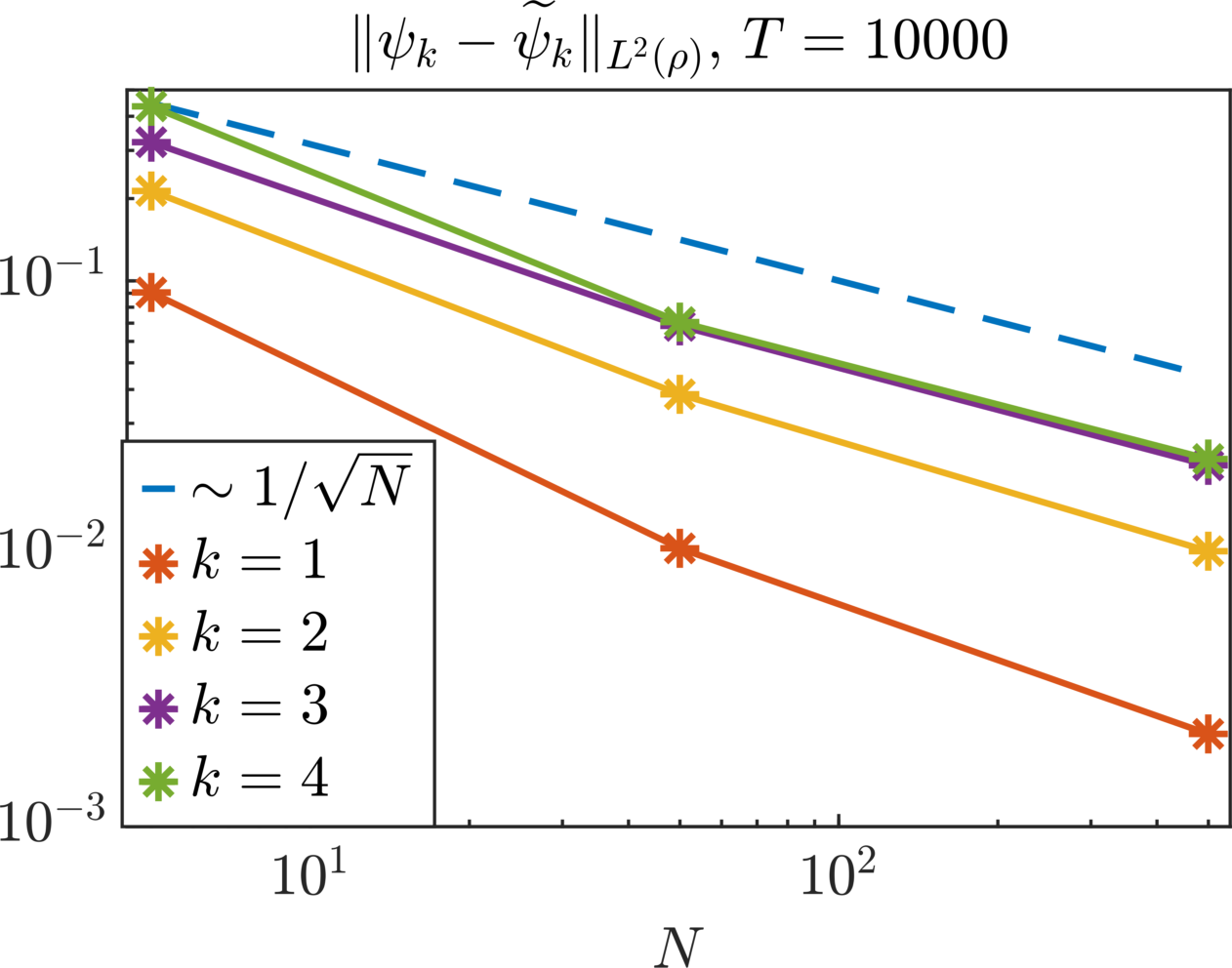} \hspace{0.5cm}
\includegraphics{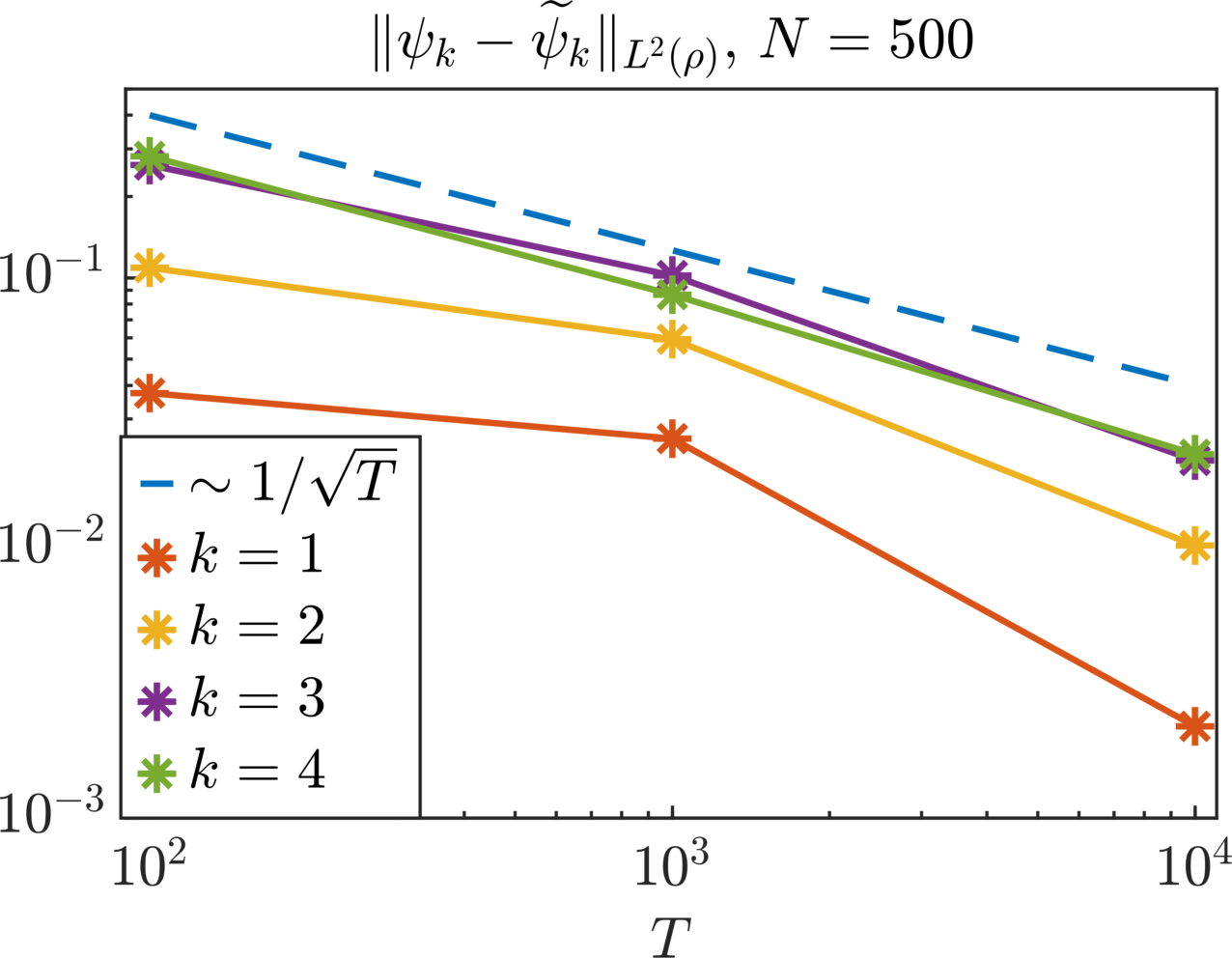}
\end{center}
\caption{Comparison between the theoretical and empirical rate of convergence of the first four (excluding the constant function) orthogonal polynomials with respect to the invariant measure $\rho$ of the mean-field Ornstein--Uhlenbeck process in $L^2(\rho)$, for both the number of particles $N$ (left) and the final time $T$ (right).}
\label{fig:orthogonal_polynomials_rate}
\end{figure}

\begin{figure}[t]
\begin{center}
\includegraphics{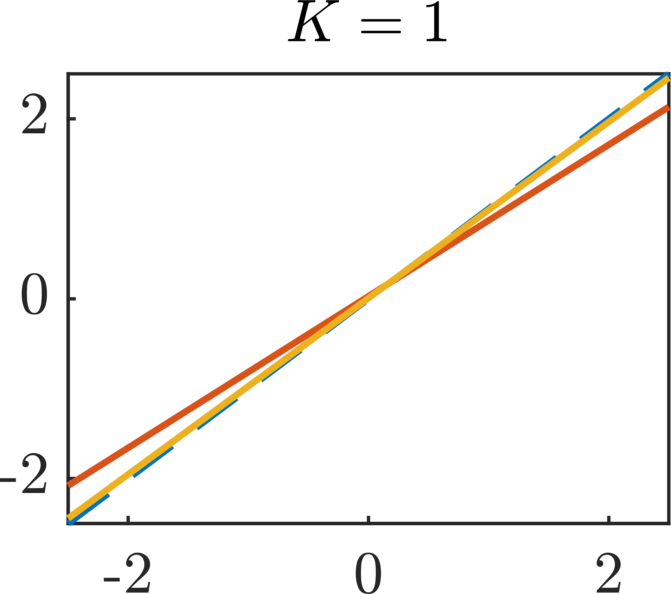}
\includegraphics{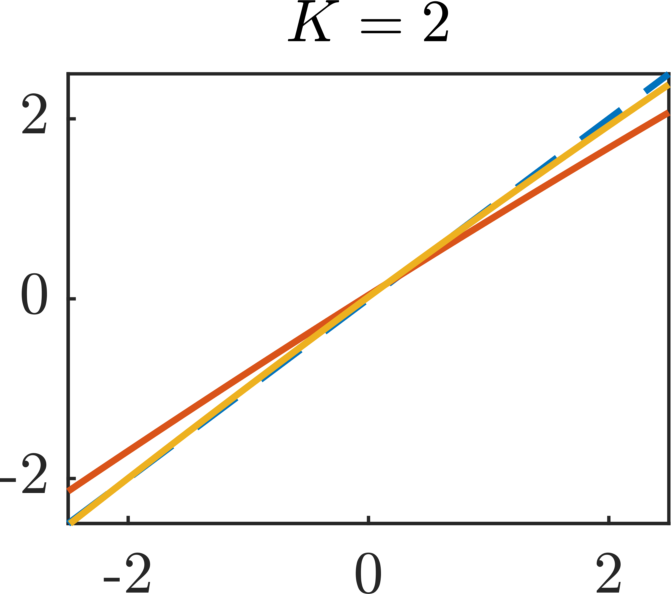}
\includegraphics{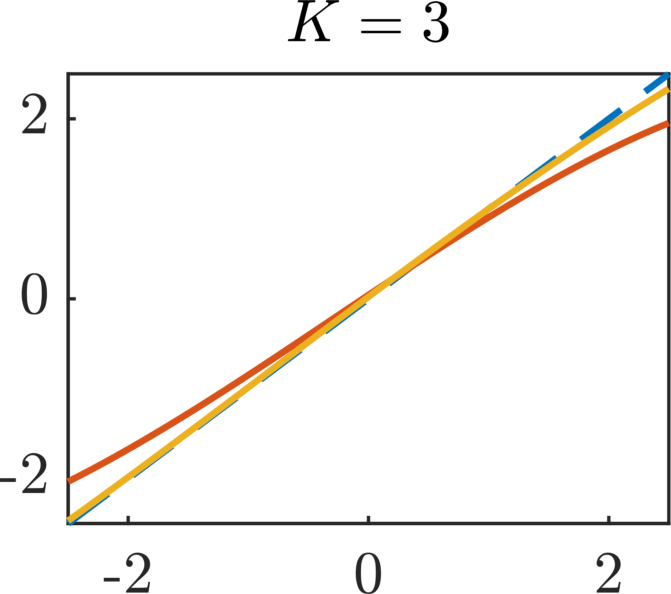}
\includegraphics{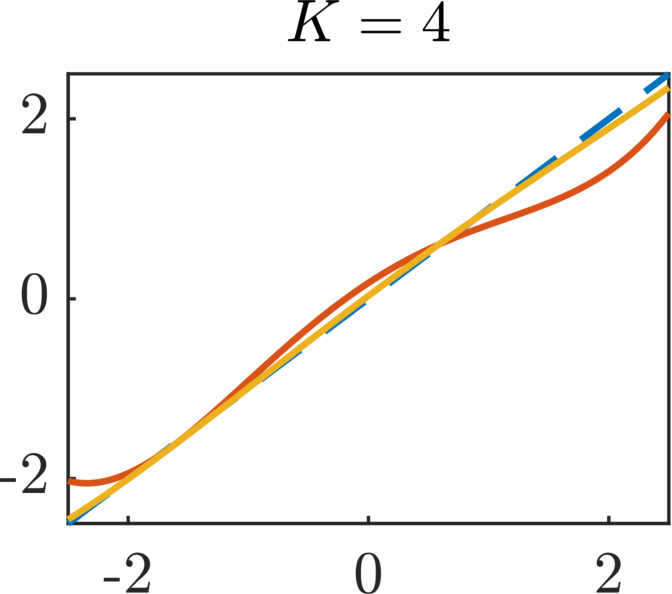} \\
\vspace{0.5cm}
\includegraphics{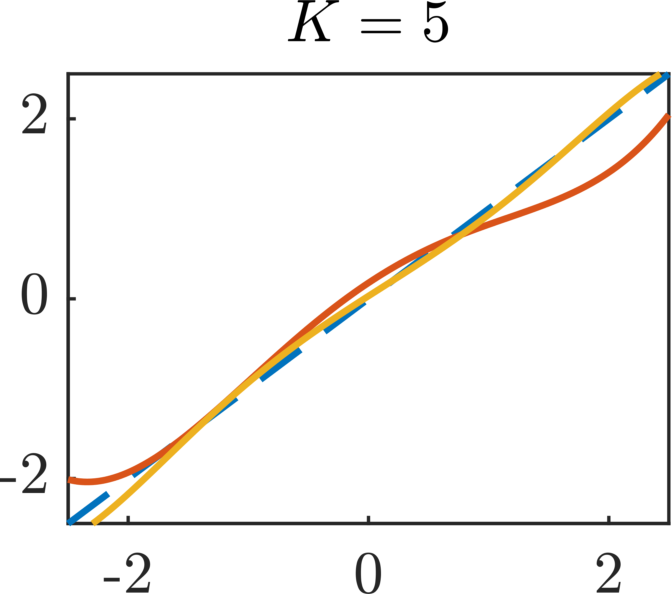}
\includegraphics{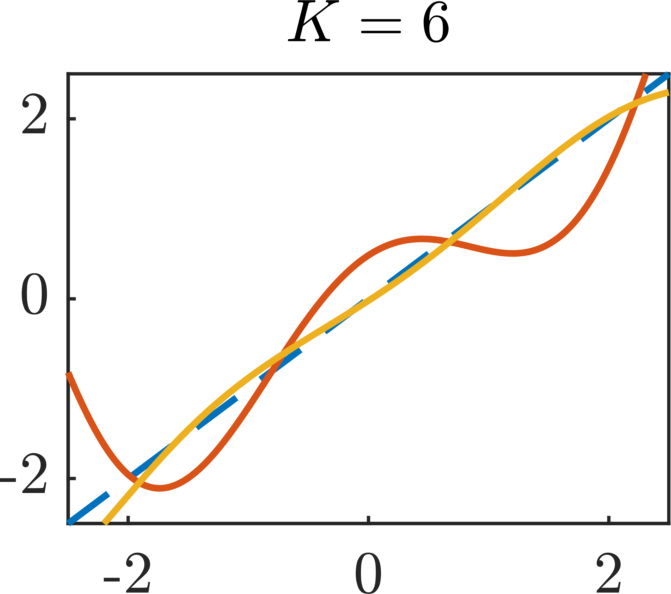}
\includegraphics{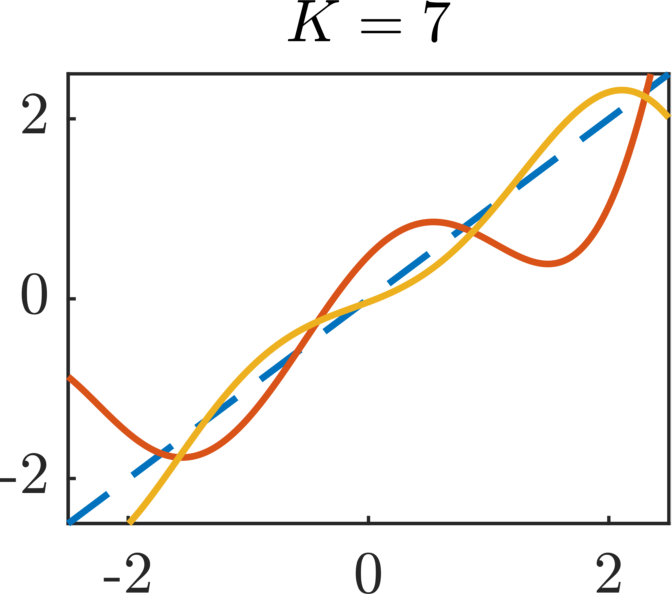}
\includegraphics{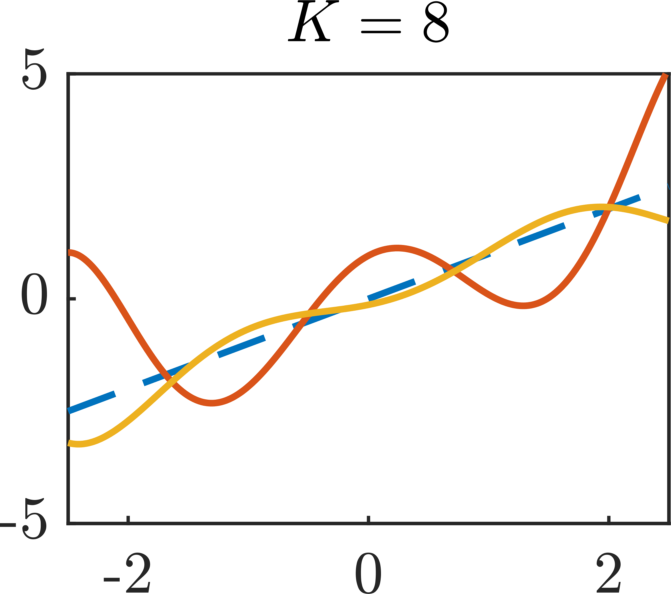} \\
\vspace{0.2cm}
\includegraphics{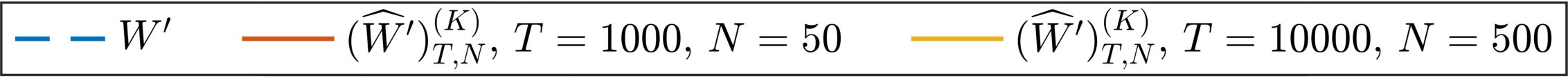}
\end{center}
\caption{Comparison between the true interaction kernel $W'$ and the estimators $(\widehat W')_{T,N}^{(K)}$ in two different cases ($T = 1\,000, N = 50$ and $T = 10\,000, N = 500$), for different numbers of Fourier coefficients $K = 1, \dots, 8$ for the Ornstein-Uhlenbeck interaction kernel.}
\label{fig:inference_OU}
\end{figure}

We consider the interacting particle system \eqref{eq:interacting_particles}, with quadratic confining and interaction potentials $V(x) = W(x) = x^2/2$, and set the diffusion coefficient $\sigma = 1$. Then, in the mean-field limit, the particle system converges to the McKean Ornstein--Uhlenbeck SDE
\begin{equation}
\d X_t = - X_t \dd t - (X_t - \E[X_t]) \dd t + \sqrt2 \d B_t,
\end{equation}
which has unique invariant measure $\mathcal N(0, 1/2)$ with density
\begin{equation}
\rho(x) = \frac1{\sqrt\pi} e^{-x^2},
\end{equation}
whose moments are given by
\begin{equation}
\mathbb M^{(k)} = \begin{cases}
0 & \text{if $k$ is odd}, \\
\left(\frac1{\sqrt2}\right)^k (k-1)!! & \text{if $k$ is even}.
\end{cases}
\end{equation}
Notice that, in this case, the orthogonal polynomials with respect to $\rho$ have the closed-form expression
\begin{equation} \label{eq:Hermite}
\psi_k(x) = \frac1{\sqrt{2^k k!}} H_k(x),
\end{equation}
for all $k \in \N$, where $H_k$ denotes the standard Hermite polynomial of degree $k$.

In \cref{fig:orthogonal_polynomials,fig:orthogonal_polynomials_rate} we compare the exact orthogonal polynomials of equation \eqref{eq:Hermite} with the approximated polynomials obtained using the single-particle trajectory observation, for different values of the number of particles $N$ in the system and the final observation time $T$. In particular, in \cref{fig:orthogonal_polynomials} we plot the first four polynomials (starting at $k=1$) varying $N = 5, 50, 500$ with $T = 10\,000$ fixed and then varying $T = 100, 1\,000, 10\,000$ with $N = 500$ fixed. We observe that, as expected, the approximation error improves when $T$ and $N$ are larger. Moreover, in \cref{fig:orthogonal_polynomials_rate} we compute the approximation error in the space $L^2(\rho)$ and verify the convergence rate provided by the theory. We remark that, even if the error is computed for a single observation and the theoretical rate holds in expectation, the two rates match. We also notice that the approximation error is greater for polynomials with higher degree $k$, and this is due to the constant $C(k)$ in \cref{pro:estimate_psi} which grows for larger values of $k$.

We then apply the proposed methodology to learn the interaction kernel $W'(x) = x$, and we consider two cases: ``few'' observations ($T = 1\,000$) in a small system ($N = 50$) and ``many'' observations ($T = 10\,000$) in a large system ($N = 500$). In \cref{fig:inference_OU} we compare the results for different Fourier series truncations $K = 1, \dots, 8$ used in the expansion of the interaction kernel. We note that, in this simple setting, two Fourier coefficients ($\beta_0$ and $\beta_1$) are enough to approximate $W'$ and $\delta(K) = \epsilon(K) = 0$ in the statement of \cref{thm:estimate_W}, since $\beta_k = 0$ for all $k \ge 2$. Therefore, we notice that increasing $K$ only worsens the results due to the constant $C(K)$, appearing in front of $1/\sqrt T$ and $1/\sqrt N$, which blows up for larger values of $K$. This is due to the ill-conditioning of both the linear system used to compute the estimator, specifically the fact that the condition number of the matrix $B^{(K)}$ increases with $K$, and the Gram--Schmidt algorithm applied to obtain the orthonormal basis. On the other hand, we observe that if the time of observation and the number of particles in the system are larger, then polynomials of higher degree still provide accurate approximations of the interaction kernel, and this is in agreement with the theoretical result in \cref{thm:estimate_W}. We emphasize that this numerical experiment shows the importance of the choice of $K$ in the method, in case we have limited observations and/or small interacting particle systems, since numerical instability becomes immediately apparent in such cases. Therefore, it would be interesting to determine the criteria for automatically adjusting $K$, and we will return to this problem in future work.

\subsection{Discrete-time, low-frequency observations}

\begin{figure}[t]
\begin{center}
\includegraphics[scale=0.98]{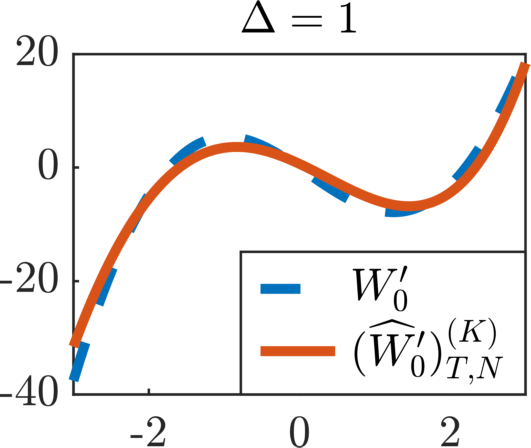}
\includegraphics[scale=0.98]{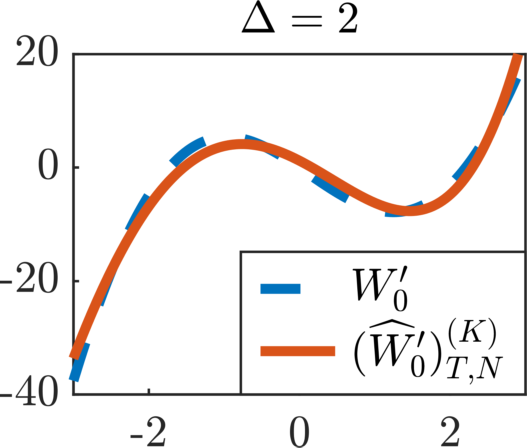}
\includegraphics[scale=0.98]{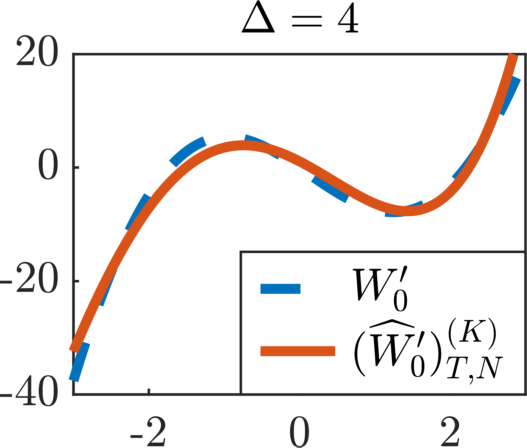}
\includegraphics[scale=0.98]{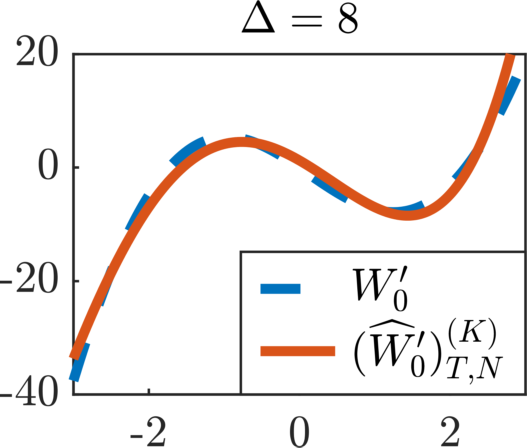}
\end{center}
\caption{Comparison between the true interaction kernel $W_0'$ from equation \eqref{eq:W0} and the estimator $(\widehat W_0')_{T,N}^{(K)}$, for the case of discrete-time observations with different sampling rates $\Delta = 1, 2, 4, 8$.}
\label{fig:discrete}
\end{figure}

As highlighted in \cref{rem:discrete}, the methodology developed in this work can also be applied when only discrete-time observations are available. In this section, we present a numerical test case to illustrate this setting. We consider a system with diffusion coefficient $\sigma = 1$, a quadratic confining potential $V(x) = x^2/2$, and a polynomial interaction potential with sinusoidal term given by
\begin{equation} \label{eq:W0}
W_0(x) = \frac{x^4}{4} - \frac{x^3}{3} + \frac{x^2}{2} + 10\cos(x).
\end{equation}
We set the final time to $T = 5000$, the number of particles to $N = 250$, and the number of Fourier coefficients to $K = 4$.

We assume that only discrete-time samples $\{ \widetilde Y_i \}_{i=0}^I$ from a single particle trajectory of the interacting particle system (as defined in equation \eqref{eq:observations}) are available, for different sampling intervals $\Delta = 1, 2, 4, 8$. The corresponding estimators $(\widehat W'_i)_{T,N}^{(K)}$ are shown in \cref{fig:discrete}, where they are compared against the true interaction kernel. We observe that the estimator accurately recovers the interaction kernel for all tested values of $\Delta$. In particular, the quality of the reconstruction does not appear to degrade significantly with larger sampling intervals, indicating robustness of the method to low-frequency data.

\subsection{Inference of interaction kernels}

\begin{figure}[t]
\begin{center}
\includegraphics{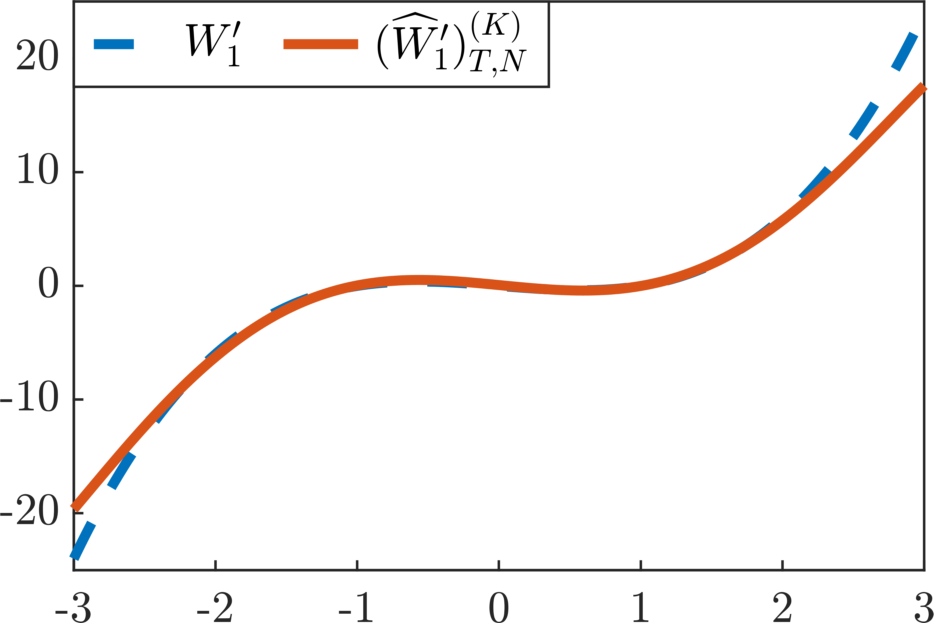} \hspace{0.5cm}
\includegraphics{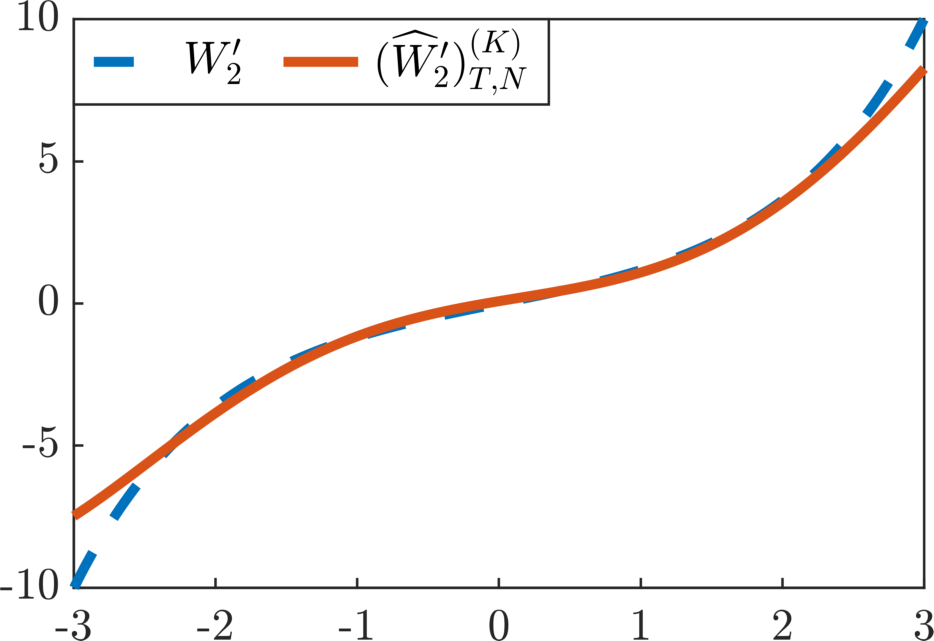} \\
\vspace{0.5cm}
\includegraphics{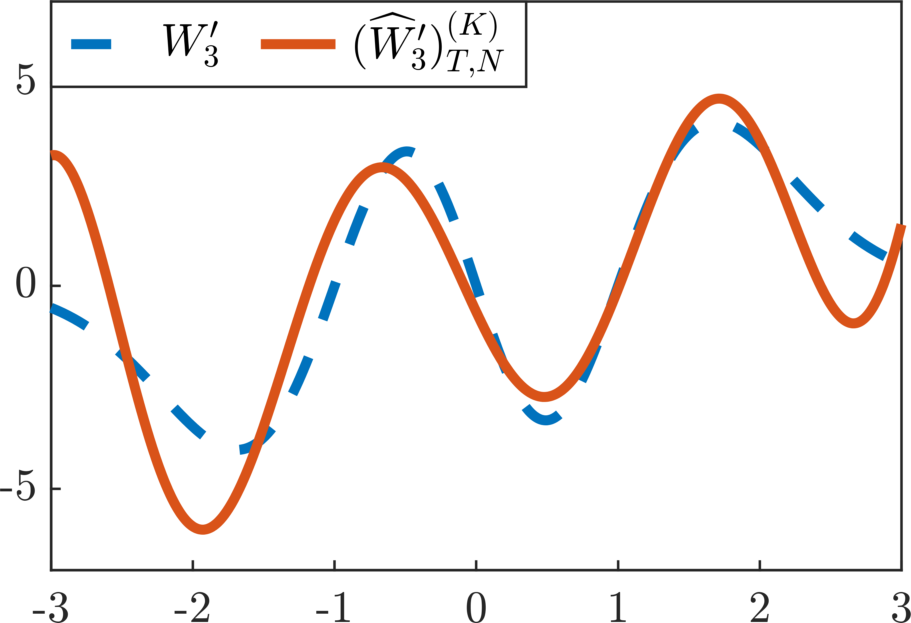} \hspace{0.5cm}
\includegraphics{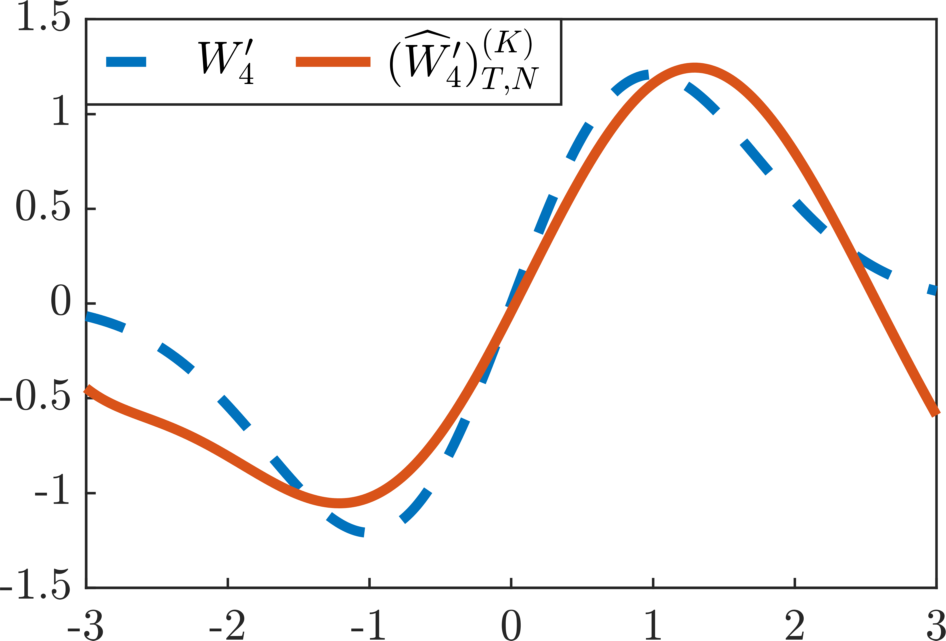} \\
\end{center}
\caption{Comparison between the true interaction kernels $W_i'$ and the estimators $(\widehat W_i')_{T,N}^{(K)}$ for the different test cases $i = 1, \dots, 4$ in equation \eqref{eq:potentials}.}
\label{fig:inference_functions}
\end{figure}

We now consider more complex interaction kernels. To ensure erogidicity, we fix the quadratic confining potential $V(x) = x^2/2$ and set the diffusion coefficient $\sigma = 1$. In the numerical experiments we will consider the following interaction potentials:
\begin{equation} \label{eq:potentials}
\begin{aligned}
W_1(x) &= \frac{x^4}4 - \frac{x^2}2, \\
W_2(x) &= \cosh(x), \\
W_3(x) &= D \left( 1 - e^{-a(x^2 - r^2)} \right)^2, \\
W_4(x) &= - \frac{A}{\sqrt{2\pi}} e^{-\frac{x^2}2}.
\end{aligned}
\end{equation}
All of these interaction potentials give rise to multiple stationary states at low temperatures. We note that $W_3$ is somewhat similar to the Morse potential, where we replaced the radius with $x^2$, and it is still the difference between an attractive and a repulsive potential, and $W_4$ represents a Gaussian attractive interaction. Moreover, we set $D = 5$, $a = 0.5$, $r = 1$ in $W_3$, and $A = 5$ in $W_4$. We consider $N = 250$ particles and set the final observation time at $T = 5000$ for the first two kernels, while we choose $N = 500$ and $T = 10\,000$ for $W_3$ and $W_4$. We then compute the estimators $(\widehat W'_i)_{T,N}^{(K)}$ for all $i = 1, \dots, 4$, and report the results in \cref{fig:inference_functions}. In the first test case, where the interaction kernel $W_1'$ is polynomial, the approximation obtained with the Fourier coefficients $K = 5$ is accurate, especially in the interval $[-2,2]$ where there are more observations. Similar considerations can be made for the second test case, where we still use $K = 5$ Fourier coefficients, even if the function $W_2'$ to be inferred is not polynomial. The third and fourth examples are more challenging, as they behave differently at infinity. In fact, both $W_3'$ and $W_4'$ vanish for $x \to \pm \infty$, and therefore they cannot be well approximated by any polynomial of finite degree. We observe that even if the estimators, which are computed with the $K = 11$ and $K = 9$ Fourier coefficients, respectively, are not as accurate as in the previous test cases, in particular for larger values of $x$, they are still able to match the overall shape of the interaction kernels. These numerical experiments show the potentiality of the approach presented in this work, which allows us to obtain a reasonable reconstruction of the interaction kernel using only a finite single trajectory from the interacting particle system, even for more complex scenarios that do not fit in the theoretical analysis.

\section{Conclusion} \label{sec:conclusion}

In this work, we proposed a methodology for learning the interaction kernel in interacting particle systems that relies only on the observation of a single particle. Our approach is based on a Fourier expansion of the interaction kernel, where the basis is made of orthogonal polynomials with respect to the invariant measure of the mean-field dynamics. We first approximated the moments of the invariant measure using the available observations, and then we employed the empirical moments to estimate the orthogonal polynomials. Finally, the Fourier coefficients are inferred by solving a linear system which still depends on the empirical moments and whose equations are derived from the stationary Fokker--Planck equation.

Our approach is easy to implement, since it only requires the approximation of the moments of the invariant measure and the solution of a low-dimensional linear system, and computationally cheap. Moreover, because of its versatility, it can infer complex interaction kernels. On the other hand, its main limitation, illustrated in the convergence analysis, is the dependence of the approximation error on the number $K$ of Fourier coefficients used in the expansion of the interaction kernel. In fact, larger values of $K$, which in principle should provide better estimates, can potentially lead to worse results if the observed trajectory or the number of particles in the system are not sufficiently large. The optimal choice of $K$ depends on multiple interacting factors, including the regularity of the function $W'$ to be approximated, the observation time $T$, the number of particles $N$, and numerical aspects such as the conditioning of the associated linear system and the stability of the Gram--Schmidt procedure. Therefore, developing and testing adaptive, data-driven procedures for automatically selecting $K$, such as the criterion inspired by the classical bias–variance trade-off in \cite[Section 5]{CGL24}, could significantly improve the robustness and practical applicability of our methodology.

The work presented in this paper can be extended in several other directions. First, we would like to improve the convergence result by quantifying the dependence on the number $K$ of Fourier coefficients, as this could help in finding techniques to reduce the approximation error. Moreover, we believe that it should be possible to define large classes of diffusion processes for which \cref{as:inverse_ck} holds and $\delta(K) \to 0$ in \cref{thm:estimate_W}, such that the asymptotic unbiasedness of the estimator is guaranteed a priori by a theoretical result. Another interesting development would be lifting the regularity hypotheses in \cref{as:L2_rho} on the confining and interaction potentials and considering the less regular functions to be estimated. In this paper, we consider the problem in one dimension in space. Similarly to the method of moments studied in~\cite{PaZ24}, we expect that the methodology developed in this paper could be extended to the multidimensional case. Conceptually, we believe that the framework of representing the interaction kernel in terms of an orthogonal basis with respect to the invariant measure of the mean-field dynamics remains applicable. Moreover, the construction of the linear system for computing the Fourier coefficients of the expansion could also be extended relatively straightforwardly. Nevertheless, moving to high-dimensional particle systems introduces, in addition to a more complicated theoretical analysis, several technical challenges that make this extension nontrivial. First, the construction of the orthogonal polynomial basis with respect to the invariant measure, which plays a central role in our estimator, becomes significantly more involved in higher dimensions. In particular, the initial basis of monomials would need to be replaced by tensor products of monomials, whose number grows exponentially with the dimension due to the curse of dimensionality. This leads to a more elaborate orthogonalization procedure, where the Gram--Schmidt algorithm might need to be replaced by more stable methods such as QR factorization with Householder reflections. Moreover, the increase in basis dimension implies larger linear systems to solve for the Fourier coefficients, which in turn affects both the conditioning of the system and the computational cost. The detailed analysis of the multidimensional problem will be presented elsewhere. To move toward more realistic applications, we will need consider (trajectory) observations that are contaminated by noise. This could be done, in principle, by combining our inference methodology with techniques from filtering, or by setting up the inference problem as a Bayesian inverse problem, as done in~\cite{Nic24}. We expect that incorporating observation noise will require additional steps in the inference methodology, increase the complexity of the statistical analysis, and potentially affect convergence rates. Finally, the numerical experiments suggest that \cref{as:unique} on the uniqueness of the invariant measure of the mean-field dynamics, e.g. our assumption that we have uniform propagation of chaos, is not necessary and that our semiparametric method is applicable even in the presence of multiple stationary states. This was also demonstrated numerically for the eigenfunction martingale estimator in~\cite{PaZ22}. We expect that the results from~\cite{MoR24}, combined with the linearization approach from~\cite{PaZ25} are sufficient to provide a rigorous justification of our methodology to the case where the mean-field dynamics exhibits non-uniqueness of stationary states. However, we expect that the performance of our estimator deteriorates as we approach the critical noise strength where the phase transition occurs. In fact, as shown in \cite{MoR24}, at the critical point the convergence to the invariant measure is only algebraic and the empirical time averages may converge more slowly. Finally, it would be interesting to apply our method to the kinetic Langevin / hypoelliptic setting~\cite{AmP24,IBG25,IgB25}. All these interesting problems will be studied in future work.

\subsection*{Acknowledgements}

We thank the anonymous reviewers whose comments and suggestions helped improve and clarify this manuscript. GAP is partially supported by an ERC-EPSRC Frontier Research Guarantee through Grant No. EP/X038645, ERC Advanced Grant No. 247031 and a Leverhulme Trust Senior Research Fellowship, SRF$\backslash$R1$\backslash$241055. AZ is supported by ``Centro di Ricerca Matematica Ennio De Giorgi'' and the ``Emma e Giovanni Sansone'' Foundation and is a member of INdAM-GNCS.

\bibliographystyle{siamnodash}
\bibliography{biblio}

\end{document}

%% file: ex_shared.tex
\usepackage[T1]{fontenc}
\usepackage{lmodern}
\usepackage[utf8]{inputenc}
\usepackage{microtype}
\usepackage{framed}
\usepackage{listings}
\usepackage{vmargin}
\usepackage{setspace}
\usepackage{mathrsfs, mathenv}
\usepackage{amsmath, amsthm, amssymb, amsfonts, amscd}
\usepackage{graphicx}
\usepackage{epstopdf}
\usepackage[svgnames]{xcolor}
\usepackage{hyperref}
\hypersetup{citecolor=blue, colorlinks=true, linkcolor=black}
\usepackage[ruled, vlined]{algorithm2e}
\usepackage[capitalise]{cleveref}
\usepackage{subcaption}
\usepackage{bbm}
\usepackage{cite}
\usepackage{verbatim}
\usepackage{pgfplots}
\usepackage{tikz}
\usetikzlibrary{arrows,decorations.pathmorphing,backgrounds,positioning,fit,matrix}
\usepackage{etoolbox}
\usepackage{color}
\usepackage{lipsum}
\usepackage{ifthen}
\usepackage[title]{appendix}
\usepackage{autonum}
\usepackage{accents}
\usepackage{xpatch}
\usepackage[]{algpseudocode}
\usepackage{multicol}
\usepackage[abs]{overpic}
\usepackage{eqnarray}

\crefname{assumption}{Assumption}{Assumptions}
\crefname{figure}{Figure}{Figures}

\theoremstyle{plain}
\newtheorem{theorem}{Theorem}[section]

\newtheorem{lemma}[theorem]{Lemma}
\newtheorem{proposition}[theorem]{Proposition}

\numberwithin{equation}{section}

\theoremstyle{definition}

\theoremstyle{remark}
\newtheorem{remark}[theorem]{Remark}
\newtheorem{assumption}[theorem]{Assumption}

\ifpdf
  \DeclareGraphicsExtensions{.eps,.pdf,.png,.jpg}
\else
  \DeclareGraphicsExtensions{.eps}
\fi

\usepackage{mathtools}

\usepackage{booktabs}

\usepackage{pgfplots}
\usepackage{tikz}
\usetikzlibrary{patterns,arrows,decorations.pathmorphing,backgrounds,positioning,fit,matrix}
\usepackage[labelfont=bf]{caption}
\usepackage{here}
\usepackage[font=normal]{subcaption}

\usepackage{enumitem}
\setlist[itemize]{leftmargin=.5in}
\setlist[enumerate]{leftmargin=.5in,topsep=3pt,itemsep=3pt,label=(\roman*)}

\newcommand{\email}[1]{\href{#1}{#1}}
\newcommand{\TheTitle}{A Fourier-based inference method for learning interaction kernels in particle systems}
\newcommand{\TheAuthors}{}
\title{\TheTitle}
\author{Grigorios A. Pavliotis \thanks{Department of Mathematics, Imperial College London, London SW7 2AZ, UK, \email{g.pavliotis@imperial.ac.uk}.}
\and Andrea Zanoni \thanks{Centro di Ricerca Matematica Ennio De Giorgi, Scuola Normale Superiore, Pisa, Italy, \email{andrea.zanoni@sns.it}.}}
\date{}

\usepackage{amsopn}

\newcommand{\abs}[1]{\left\lvert#1\right\rvert}
\newcommand{\norm}[1]{\left\|#1\right\|}
\renewcommand{\Pr}{\mathbb{P}}

\newcommand{\N}{\mathbb{N}}

\newcommand{\R}{\mathbb{R}}

\newcommand{\epl}{\varepsilon}

\newcommand{\defeq}{\coloneqq}
\newcommand{\eqdef}{\eqqcolon}

\newcommand{\E}{\operatorname{\mathbb{E}}}

\DeclareMathOperator*{\argmin}{arg\,min}

\newcommand{\compl}{\mathsf{C}}

\renewcommand{\d}{\mathrm{d}}
\newcommand{\dd}{\,\mathrm{d}}
\definecolor{shade}{RGB}{100, 100, 100}
\definecolor{bordeaux}{RGB}{128, 0, 50}

\usepackage[usestackEOL]{stackengine}

\definecolor{leg1}{RGB}{0,114,189}
\definecolor{leg2}{RGB}{217,83,25}
\definecolor{leg3}{RGB}{237,177,32}
\definecolor{leg4}{RGB}{126,47,142}
\definecolor{leg5}{RGB}{119,172,48}

\definecolor{leg21}{RGB}{62,38,169}
\definecolor{leg22}{RGB}{46,135,247}
\definecolor{leg23}{RGB}{55,200,151}
\definecolor{leg24}{RGB}{254,195,56}

\ifpdf
\hypersetup{
	pdftitle={\TheTitle},
	pdfauthor={\TheAuthors}
}
\fi